\documentclass[11pt]{article}

\usepackage{fullpage}
\usepackage{subcaption}

\usepackage{enumerate}
\usepackage{amsmath,amsfonts,amssymb,amsthm, bbold}

\usepackage{graphics,esint}
\usepackage{epsfig}
\usepackage{xcolor}
\usepackage{xspace}
\definecolor{pingreen}{rgb}{0,39,14}

\usepackage{xfrac}
\usepackage{comment}
\usepackage{hyperref}
\usepackage{cleveref}

\crefname{section}{§}{§§}
\Crefname{section}{§}{§§}

\makeatletter
\newtheorem*{rep@theorem}{\rep@title}
\newcommand{\newreptheorem}[2]{%
\newenvironment{rep#1}[1]{%
 \def\rep@title{#2 \ref{##1}}%
 \begin{rep@theorem}}%
 {\end{rep@theorem}}}
\makeatother

\newtheorem{theorem}{Theorem}[section]

\newtheorem{lemma}[theorem]{Lemma}

\newtheorem{remark}[theorem]{Remark}

\newtheorem{proposition}[theorem]{Proposition}

\newtheorem{corollary}[theorem]{Corollary}

\newreptheorem{theorem}{Theorem}

\newcommand{\R}{\mathbb{R}}

\newcommand{\Fcal}{\mathcal{F}}

\def\T{\mathbb{T}}

\def\N{\mathbb N}
\def\eps{\varepsilon}

\def\eps{\varepsilon}
\def\loc{\mathrm{loc}}

\def\d {\,\mathrm {d}}
\def\dx{\,\mathrm {d}x}
\def\dz{\,\mathrm {d}z}
\def\ds{\,\mathrm {d}s}

\def\dt{\,\mathrm {d}t}
\def\dy{\,\mathrm {d}y}

\def\lin{\mathrm{lin}}

\def\de0#1{\rule[3pt]{#1}{0.4pt} \hspace{-0.1pt} \rule[3.05pt]{0.05pt}{0.4pt} \hspace{-0.1pt} \rule[3.1pt]{0.05pt}{0.4pt} \hspace{-0.1pt} \rule[3.15pt]{0.05pt}{0.4pt} \hspace{-0.1pt} \rule[3.2pt]{0.05pt}{0.4pt} \hspace{-0.1pt} \rule[3.25pt]{0.05pt}{0.4pt} \hspace{-0.1pt} \rule[3.3pt]{0.05pt}{0.4pt} \hspace{-0.1pt} \rule[3.35pt]{0.05pt}{0.4pt} \hspace{-0.1pt} \rule[3.4pt]{0.05pt}{0.4pt} \hspace{-0.1pt} \rule[3.45pt]{0.05pt}{0.4pt} \hspace{-0.1pt} \rule[3.5pt]{0.05pt}{0.4pt} \hspace{-0.1pt} \rule[3.55pt]{0.05pt}{0.4pt} \hspace{-0.1pt} \rule[3.6pt]{0.05pt}{0.4pt} \hspace{-0.1pt} \rule[3.65pt]{0.05pt}{0.4pt} \hspace{-0.1pt} \rule[3.7pt]{0.05pt}{0.4pt} \hspace{-0.1pt} \rule[3.75pt]{0.05pt}{0.4pt} \hspace{-0.1pt} \rule[3.8pt]{0.05pt}{0.4pt} \hspace{-0.1pt} \rule[3.85pt]{0.05pt}{0.4pt} \hspace{-0.1pt} \rule[3.9pt]{0.05pt}{0.4pt} \hspace{-0.1pt} \rule[3.95pt]{0.05pt}{0.4pt} \hspace{-0.1pt} \rule[4.0pt]{0.05pt}{0.4pt} \hspace{-0.1pt} \rule[4.05pt]{0.05pt}{0.4pt} \hspace{-0.1pt} \rule[4.1pt]{0.05pt}{0.4pt} \hspace{-0.1pt} \rule[4.15pt]{0.05pt}{0.4pt} \hspace{-0.1pt} \rule[4.2pt]{0.05pt}{0.4pt} \hspace{-0.1pt} \rule[4.25pt]{0.05pt}{0.4pt} \hspace{-0.1pt} \rule[4.3pt]{0.05pt}{0.4pt} \hspace{-0.1pt} \rule[4.35pt]{0.05pt}{0.4pt} \hspace{-0.1pt} \rule[4.4pt]{0.05pt}{0.4pt} \hspace{-0.1pt} \rule[4.45pt]{0.05pt}{0.4pt} \hspace{-0.1pt} \rule[4.5pt]{0.05pt}{0.4pt} \hspace{-0.1pt} \rule[4.55pt]{0.05pt}{0.4pt} \hspace{-0.1pt} \rule[4.6pt]{0.05pt}{0.4pt} \hspace{-0.1pt} \rule[4.65pt]{0.05pt}{0.4pt} \hspace{-0.1pt} \rule[4.7pt]{0.05pt}{0.4pt} \hspace{-0.1pt} \rule[4.75pt]{0.05pt}{0.4pt} \hspace{-0.1pt} \rule[4.8pt]{0.05pt}{0.4pt} \hspace{-0.1pt} \rule[4.85pt]{0.05pt}{0.4pt} \hspace{-0.1pt} \rule[4.9pt]{0.05pt}{0.4pt} \hspace{-0.1pt} \rule[4.95pt]{0.05pt}{0.4pt} \hspace{-0.1pt} \rule[5.0pt]{0.05pt}{0.4pt} \hspace{-0.1pt} \rule[5.05pt]{0.05pt}{0.4pt} \hspace{-0.1pt} \rule[5.1pt]{0.05pt}{0.4pt} \hspace{-0.1pt} \rule[5.15pt]{0.05pt}{0.4pt} \hspace{-0.1pt} \rule[5.2pt]{0.05pt}{0.4pt} \hspace{-0.1pt} \rule[5.25pt]{0.05pt}{0.4pt} \hspace{-0.1pt} \rule[5.3pt]{0.05pt}{0.4pt} \hspace{-0.1pt} \rule[5.35pt]{0.05pt}{0.4pt} \hspace{-0.1pt} \rule[5.4pt]{0.05pt}{0.4pt} \hspace{-0.1pt} \rule[5.45pt]{0.05pt}{0.4pt} \hspace{-0.1pt} \rule[5.5pt]{0.05pt}{0.4pt} \hspace{-0.1pt} \rule[5.55pt]{0.05pt}{0.4pt} \hspace{-0.1pt} \rule[5.6pt]{0.05pt}{0.4pt} \hspace{-0.1pt} \rule[5.65pt]{0.05pt}{0.4pt} \hspace{-0.1pt} \rule[5.7pt]{0.05pt}{0.4pt} \hspace{-0.1pt} \rule[5.75pt]{0.05pt}{0.4pt} \hspace{-0.1pt} \rule[5.8pt]{0.05pt}{0.4pt} \hspace{-0.1pt} \rule[5.85pt]{0.05pt}{0.4pt} \hspace{-0.1pt} \rule[5.9pt]{0.05pt}{0.4pt} \hspace{-0.1pt} \rule[5.95pt]{0.05pt}{0.4pt} \hspace{-0.1pt} \rule[6.0pt]{0.05pt}{0.4pt}}	%fundamental brick
\numberwithin{equation}{section}

\usepackage{authblk}

\author[1]{Sara Daneri\thanks{sara.daneri@gssi.it}}
\author[2]{Emanuela Radici\thanks{emanuela.radici@epfl.ch}}
\author[3]{Eris Runa\thanks{eris.runa@gmail.com}}
\affil[1]{Gran Sasso Science Institute, L'Aquila, Italy}
\affil[2]{Ecole Polytechnique F\'ed\'erale de Lausanne, Switzerland}
\affil[3]{Deutsche Bank AG, London, UK}

  \title{Deterministic particle approximation of aggregation diffusion equations with nonlinear mobility }
  
\begin{document}
  	
  	\maketitle

  \begin{abstract}
  	We consider a class of aggregation-diffusion equations on unbounded one dimensional  domains with Lipschitz nonincreasing mobility function. We show strong $L^1$-convergence of a suitable deterministic particle approximation to weak solutions of a class aggregation-diffusion  PDEs (coinciding with the classical ones in the no vacuum regions) for any bounded initial data of finite energy. In order to prove well-posedness and convergence of the scheme with no BV or no vacuum assumptions and overcome the issues posed in this setting  by the presence of a mobility function, we improve  and strengthen the techniques introduced in \cite{DRR}.
  \end{abstract}

  \section{Introduction}
  In this paper we consider  the following aggregation-diffusion equation with nonlinear mobility $v$
  \begin{equation}
  \label{eq:mainPDE}
  \partial_t \rho = \partial_x\big(\rho v(\rho) \partial_x (K \ast \rho) +\partial_x \phi_v(\rho) \big), \qquad t\in[0,T],\, x\in \R.
  \end{equation}

  By  $K$ we denote an interaction kernel satisfying the following conditions
  \begin{align}
  	&K(z)=K(|z|);\label{eq:k1}\\
  	&K\in C^0(\R)\cap C^2(\R\setminus\{0\});\label{eq:k2}\\
  	&K,\,K'\,\in L^\infty(\R),\, K''_{|_{\R\setminus\{0\}}} \in L^{\infty}(\R\setminus\{0\});\label{eq:k3}\\
  	&\|K'\|_{L^1(\R)}<\infty,\label{eq:k4}
  \end{align}
by $v \in \mathrm{Lip}([0,\infty), [0,\infty))$ a velocity term preventing the overcrowding effect  and satisfying the following 
\begin{align}
&\| v \|_{\infty} = v(0) < \infty; \label{eq:v1} \\
&v' \leq 0; \label{eq:v2}%\\
%&v(z)\to 0 \quad\text{ as $z\to+\infty$}\label{eq:v3}
%&supp\, v \in [0,R] \text{ for some $R>0$,}\label{eq:v3}
\end{align}  
and by $\phi_v$ a $C^1$ function defined by
\begin{equation}
	\label{eq:phivdef}
	\phi_v(s)=\int_0^s\xi W''(\xi)v(\xi)\d\xi,
\end{equation} 
where either
\begin{equation}
	W(\rho)=\rho\log\rho
\end{equation}
or 
\begin{equation}
	W(\rho)=\frac{\rho^m}{m-1},\quad m>1.
\end{equation}
%\begin{align}
%&{W \geq 0}, \quad W(0)=0\label{eq:phi1}\\
%&W\text{ strictly convex}.\label{eq:phi2}
%&\phi\text{ strictly monotone increasing }\label{eq:phi3}\\
%&\phi'(\rho)\rho\leq c_0\phi(\rho)\label{eq:phi4}\\ &\phi(\rho)\leq\max\{\rho, c_0W(\rho)\}\label{eq:phi5}
%\end{align}
%for some constant $c_0>0$ and there exist $c_1,c_2>0$ such that
%\begin{equation}\label{eq:phirhocond}
%\phi(\rho)\leq c_1\rho\qquad\text{if $\rho\leq \hat{\rho}<1$},\qquad\text{and} \qquad\phi(\rho)\geq c_2\rho\qquad\text{if $\rho\geq \bar{\rho}>1$}.
%\end{equation} 
Under the above assumptions on $W$, the function 
\begin{equation}
	\label{eq:defphi}
	\phi(s)=\phi_1(s)=\int_0^s\xi W''(\xi)\d\xi
\end{equation}
satisfies the assumptions in \cite{DRR}. In particular, $\phi$ is given by $\phi(\rho)=\rho^m$ for $m\geq1$.

Notice that $K$ is not assumed to have a definite attractive or repulsive behaviour and might have a Lipschitz singularity at the origin. 

In the case of constant mobility (e.g. $v\equiv1$) the PDE \eqref{eq:mainPDE} can be seen as the gradient flow of the following functional
\begin{equation}
\Fcal(\rho) :=
\frac12\int_{\R}\int_{\R} K(x-y) \rho(x)\rho(y)\dx\dy +\int_{\R}W(\rho(x))\dx.
\end{equation}

Such formulation carries out a naturally induced Lagrangian description in the pseudo-inverse formalism.  This approach was introduced for pure diffusion equations in \cite{CT}; while several numerical properties were studied in \cite{R1,R2} and also in \cite{GT} for the problem coupled with smooth attractive kernels.  A first convergence result for the Lagrangian approximation was obtained in \cite{MO} for nonlinear diffusion equations in bounded domains with zero velocity boundary condition and strictly positive densities with bounded variation.  In \cite{MS}, instead,  the aggregation-diffusion regime is considered and the authors can prove that the deterministic approximating particle scheme converges to a weak solution of the pseudo-inverse equation.  
More recently,  in \cite{DRR},  the authors of this manuscript consider the same particle scheme as in \cite{MS} and obtain a convergence result at the level of the continuity equation for the density for singular nonlocal interaction potentials.  Fundamental novelties of \cite{DRR} are that no restrictions to bounded domains neither the \emph{no-vacuum} condition are required and,  moreover,  the initial datum is not assumed to have bounded variation.  The compactness proved in \cite{DRR} can be achieved exploiting the gradient flow structure of \eqref{eq:mainPDE} when $v\equiv1$.  Indeed, it is possible to show that the energy $\Fcal$ remains uniformly bounded along the many particle limits once the initial datum has itself finite energy,  and such bound provides both the desired $L^1$ compactness and a maximum principle for the macroscopic density. 
A different approach called \emph{blob-method}, still based on deterministic particle scheme,  was introduced in \cite{CCP} to approximate weak solutions of the multidimensional PDE \eqref{eq:mainPDE} for $v=1$, $C^2$ nonlocal interactions and quadratic diffusion.

\vskip 0.2 cm
{Although equations of the type \eqref{eq:mainPDE} are generally
not gradient flows with respect to the standard Wasserstein metric when $v \neq 1$,  they can be interpreted as gradient-flow evolutions of $\Fcal$ with respect to a generalised Wasserstein distance which takes into account the non-linearity.  Such point of view was introduced in \cite{DNS} and \cite{LM} and further investigated in \cite{CLSS}, while some applications of the corresponding gradient-flow structure may be found in the study
of non-linear cross-diffusion systems \cite{BDFPS} and fourth-order equations, as Cahn-Hillard and thin films equations (see \cite{LMS}, \cite{Zin}).
However, being the generalised Wasserstein metric only expressed in a dynamical implicit formulation, it is not easy to provide the corresponding generalization of the Jordan-Kinderlehrer-Otto scheme and construct solutions of \eqref{eq:mainPDE} in this spirit.  The validation of a scalar transport equation with nonlinear mobility,  which can be also regarded as a scalar nonlinear conservation law,  can then be obtained proving the convergence of systems of fully deterministic interacting particles.  A first rigorous result dealing with pure linear drift can be found in \cite{difrancesco-rosini},  while the case of external potentials with symmetries has been considered in  \cite{DFS}.  
A deterministic many-particle limit result for scalar evolutions with non-local interactions was initially performed in \cite{DFR} in the smooth and pure aggregative regime.  In \cite{FT} and \cite{RS} the case of mildly singular non-local interaction kernels (as Morse and Yukawa type ones) is considered.  
The authors of \cite{FT} work directly at the level of the gradient-flow structure at the continuous time-scale exploiting an energy-dissipation balance argument,  while more general time-dependent kernels are included in the analysis in \cite{RS} thanks to the introduction of a new integrated version of the deterministic particle scheme. 
On the other hand, the first deterministic many-particle limit for scalar non-local transport equations as in \eqref{eq:mainPDE} with nonlinear mobility and a diffusive term was considered in \cite{fagioli-radici-1}.  The compactness argument employed in \cite{fagioli-radici-1} resembles the techniques used in \cite{DFR} and it is based on strong $BV$ estimates and a good control on the support of the moving particles.  Since the presence of the diffusive term in \eqref{eq:mainPDE} plays against the latter bound,  the analysis in \cite{fagioli-radici-1} is carried out inside bounded intervals with zero velocity boundary conditions and for initial data enjoying the \emph{no vacuum} condition.   
} 
\vskip 0.2 cm

In this paper we follow the general approach proposed in \cite{DRR} for constant mobility functions, which allows to show a strong convergence of the deterministic particle scheme for general bounded initial data of finite energy on unbounded domains. Namely, we first consider piecewise constant approximations on tori of size $L>0$, then show strong $L^1$-convergence of the scheme for fixed $L$, and finally thanks to the fact that our $L^1$-compactness estimates are independent of $L$ and of lower bounds on the initial density, we can take limits of the scheme as $L\to+\infty$ and recover an approximation procedure on the whole real line. 

As in \cite{DRR}, the $L^1$-compactness is guaranteed by $W^{1,2}$ estimates for the nonlinear term $\phi_v(\rho)$, which are in turn derived from a controlled growth of a suitable energy functional along the deterministic particle flow and an explicit computation of its time derivative along the flow. 

However, the fact that we do not assume either bounds on the BV norm or lower bounds of the initial density lead, in the case of nonconstant mobility, to solutions of a PDE which possibly differs from \eqref{eq:mainPDE} on the vacuum regions for the limit density. This is mainly due to the fact that the distance between subsequent particles of the scheme may not go to zero as the number of particles tends to infinity, due to the presence of a nontrival nonlinear mobility. Moreover, the presence of a nonlinear mobility requires the use of more refined estimates which circumvent the need for BV bounds on the density (or on the mobility function) in the nonlocal part of the equation/energy, in proving both $L^\infty$ and energy bounds. 

The energy we use to derive $W^{1,2}$ bounds in this paper is the following. For any $L>0$, let
\begin{equation}
\Fcal_L^v(\rho):=\frac12\int_{\T_L}\int_{\T_L} K(x-y) \rho(x)\rho(y)\dx\dy +\int_{\T_L}W_v(\rho)\dx,
\end{equation} 
where $W_v:\R\to\R$ is such that
\begin{equation}
	\phi_v(s)=sW_v'(s)-W_v(s).
\end{equation}

We prove the following theorems.

\begin{theorem}\label{thm:linfty}
	Let $\rho_{0,L}\in L^\infty(\T_L)$, $T>0$. Then, there exists a constant $\bar C=\bar C(\Fcal^v_L(\rho_{0,L}),K, \|\rho_{0,L}\|_{L^\infty},\|v\|_{L^\infty},T)$  such that the following holds. There exists $\bar N=\bar N(L,T,\Fcal^v_L(\rho_{0,L}),\|v\|_{L^\infty})$ s.t. the deterministic particle approximations $\{\rho^N_L\}_{N\geq\bar N}$ starting from $\rho_{0,L}$ are well-defined on $[0,T]$ and moreover
	\begin{equation}
		\sup_{N\geq\bar N}\sup_{t\in[0,T)}\|\rho^N_L(t)\|_{L^\infty(\T_L)}\leq \bar C.
	\end{equation}
\end{theorem}

\begin{theorem}\label{thm:convpdelfixed} Let $L>0$, $\rho_{0,L}\in L^\infty(\T_L)$, $T>0$. Then the deterministic particle approximations $\{\rho^N_L\}_{N\geq\bar N}$ starting from $\rho_{0,L}$ converge in $L^1$ as $N\to\infty$ to a solution $\rho_L$ of the following PDE in weak form
		\begin{align}
			\int_0^T\int_{\T_L}&\rho_L(t,x)\partial_t\varphi(t,x)-\rho_L(t,x)K'\ast\rho_L(t,x)v(\rho_L(t,x))\partial_x\varphi(t,x)+{\Phi}_{v,L}(t,x)\partial_{xx}\varphi(t,x)\dx\dt=0\label{eq:pdel}
		\end{align}
	where $\Phi_{v,L}(t,x)=\phi_v(\rho_L(t,x))$ for a.e. $x\in\{\rho_L(t)>0\}$.
\end{theorem}

\begin{theorem}\label{thm:convinl}
	Consider initial data $\rho_0\in L^1(\R)\cap L^\infty(\R)$ with $\int_{\R}|x|\rho_0(x)\dx<+\infty$ and define $\rho_{0,L}$ by cutting $\rho_0$ on $[-L/2,L/2)$ and extending it periodically. Then the functions $\rho_L\chi_{[-L/2,L/2)}$, being $\rho_L$ limits of the deterministic particle approximations starting from $\rho_{0,L}$ on $\T_L$ found in Theorem~\ref{thm:convpdelfixed}, converge (up to subsequences) in $L^1([0,T]\times\R)$ to a weak solution $\rho$ of the PDE in weak form
		\begin{align}
		\int_0^T\int_{\R}&\rho(t,x)\partial_t\varphi(t,x)-\rho(t,x)K'\ast\rho(t,x)v(\rho(t,x))\partial_x\varphi(t,x)+{\Phi}_v(t,x)\partial_{xx}\varphi(t,x)\dx\dt=0,
	\end{align}
where ${\Phi}_v(t,x)=\phi_v(\rho(t,x))$ for a.e. $x\in \{\rho(t)>0\}$. 
\end{theorem}

Our main improvements w.r.t. the existing literature are the following: we consider general mobilities, with no assumptions on their rate of decay; we consider general Lipschitz  kernels, with possibly mixed attractive-repulsive behaviour; we assume neither BV bounds nor lower bounds  on the initial densities. 

The absence of a minimum principle in case of nonconstant mobilities determines the presence of the terms ${\Phi}_{v,L}$ and ${\Phi}_{v}$ on the vacuum regions $\{\rho_L=0\}$ and $\{\rho=0\}$. 

Regarding the interest for kernels with mixed attractive-repulsive behaviour satisfying our assumptions we recall the two-Yukawa kernel considered in~\cite{DR2}, which is used in biology to model pattern formation in colloidal systems 
\begin{equation}\label{eq:Yuk}
	K(z)=-\beta^2e^{-\beta |z|}+e^{-|z|},\quad \beta\gg1.
\end{equation}The $\Gamma$-limit of such a functional as $\beta\to+\infty$ (namely with local attractive term) has been characterized in~\cite{DR2}. 
In suitable regimes,  minimizers of the limit functionals have been proved to be given by periodic unions of stripes (i.e., intervals in one dimension,  see~\cite{DR2}) with techniques developed in~\cite{GR, DR1,  DR2, DKR, Ker, DR3, DR4, DR5, DR6}. For characterization of minimizers with power law attractive-repulsive potentials see~\cite{ChTo}.

\section{Preliminary facts}

  %\subsection{Pseudo-inverse}
  
 % Fix $x_0\in\T_L$. Given a nonnegative  density function $\rho:\T_L\to[0,+\infty)$ with $\int_{\T_L}\rho=c_L$, define its pseudo-inverse $X:[0,c_L]\to\T_L$ as follows
 % \begin{equation}\label{eq:pseudoinv}
 % X(z)=\sup\Big\{x:\int_{x_0}^x\rho(y)\dy<z\Big\}.
 % \end{equation}
 % If $\rho(t,x)\in L^\infty$ is a weak solution of~\eqref{eq:mainPDE} on $\T_L$ and $\rho>0$, then $X(t,z)=X_{\rho(t,\cdot)}(z)$ solves the PDE
 % \begin{equation*}
 % \partial_tX(t,z)=- v\big( (\partial_z X(t,z)^{-1} ) \big) \left( \int_0^{c_L}K'(X(t,z)-X(t,\xi))\d\xi-\partial_z\phi(\rho(X(t,z))) \right).
 % \end{equation*}

  \subsection{Deterministic particle approximation}\label{sec:det}
  
  Our goal is to approximate~\eqref{eq:mainPDE} with a moving particle approximation on a series of increasing tori.
  
  Let us fix $L>0$, $N\in\N$ and $x_0\in\T_L$.

Given $\rho_{0,L}\in L^\infty(\T_L)$ with $\int_{\T_L}\rho_{0,L}=c_L\leq1$ and $\rho_{0,L}\geq0$, for all $k=1,\dots,N$ define (for the ordering identify $\T_L$ with $[-L/2,L/2]$ and set $x_0=-L/2$)
  \begin{equation*}
  x_k=\sup\Big\{x:\int_{x_{k-1}}^x\rho_{0,L}(y)\dy<\frac{c_L}{N}\Big\}.
  \end{equation*}
  Notice that $x_0<\cdots<x_{N-1}$  and $x_N=L/2=x_0+L$ (resp. $x_N=x_0$ on $\T_L$).

  For every $k=0,\dots,N$ define the following system of ODEs on $\T_L$
  \begin{align}\label{eq:ode}
  \dot x_k(t)=&-\frac{c_L }{N} v_k(t)\sum_{j\neq k}K'(x_k(t)-x_j(t))-\frac{ N}{c_L}G_k(t),
  \end{align}
  where
  \begin{align*}
  \rho_k(t)&=\frac{c_L}{N(x_{k+1}(t)-x_k(t))}, \notag\\
   G_k(t) &= \phi_v(\rho_k(t))-\phi_v(\rho_{k-1}(t)),\notag\\
    v_k(t) &=v(\max(\rho_k(t),\rho_{k-1}(t)))\label{eq:v_k}%\text{ s.t. }G_k(t)=F_k(t)v_k(t),	\text{ where }F_k(t)=\phi(\rho_k)-\phi(\rho_{k-1}) 
\end{align*}
  and with initial conditions
  \begin{equation*}
  x_k(0)=x_k,
  \end{equation*}
  as long as $x_0(t)<\cdots<x_{N-1}(t)$. 
In the above the dependence on $N,L$ is omitted since clear from the context. Notice that, since $v$ is nonincreasing, 
\begin{equation}\label{eq:vk}
v_k(t)=\min(v(\rho_k(t)),v(\rho_{k-1}(t))).
\end{equation}Then define the deterministic particle approximations starting from $\rho_{0,L}$ as the piecewise constant functions 
  \begin{equation}\label{eq:rhonl}
  \rho^N_L(t,x)=\sum_{k=0}^{N-1}\rho_k(t)\chi_{[x_k(t),x_{k+1}(t))}(x).
  \end{equation}
  We say that the deterministic particle approximation $\rho^N_L$ is well defined on $[0,T)$ provided the relation $x_0(t)<\dots<x_{N-1}(t)$ holds for all $t\in[0,T)$. 
  Notice that, despite of the above definition of starting point of the deterministic particle approximation, $\rho_{0,L}\neq\rho^N_L(0)$ {but $\rho^N_L(0)\to\rho_{0,L}$ in $L^1(\T_L)$ as $N\to\infty$}.
  
  Let $M_{\rho^N_L(t)}:\T_L\to[0,c_L]$ be the cumulative distribution function of $\rho^N_L(t,\cdot)$, namely
  \[
  M_{\rho^N_L(t)}(x)=\int_{x_0}^x \rho^N_L(t,y)\dy
  \]
   and $X(t,\cdot):[0,c_L]\to\T_L$ its pseudoinverse function, i.e.  
   \begin{equation}\label{eq:pseudoinv}
   	X(t,z)=\sup\Big\{x:\int_{x_0}^x\rho^N_L(t,y)\dy<z\Big\}.
   	\end{equation}
   Notice that by definition $x_k(t)=X(t,kc_L/N)$. 
   For simplicity of notation, let us also define 
   \begin{equation}
   	\label{eq:vnl}
   	v^N_L(t,x)=\sum_{k=0}^{N-1}v_k(t)\chi_{[x_k(t),x_{k+1}(t))}(x).
   \end{equation}
      Let moreover 
   \begin{align}
   	{K^\lin}'&(\rho^N_L,v^N_L)(t,x-y)=\sum_{k=0}^{N-1}\chi_{[kc_L/N,(k+1)c_L/N)}(M_{\rho^N_L(t)}(x))\sum_{j=0}^{N-1}\chi_{[jc_L/N,(j+1)c_L/N)}(M_{\rho^N_L(t)}(y))\cdot\notag\\
   	&\cdot\Bigl[(1-\chi_{\{0\}}(x_j(t)-x_k(t)))v_k{K}'(x_k(t)-x_{j}(t))\notag\\
   	&+\Bigl(\frac{N}{c_L}M_{\rho^N_L(t)}(x)-k\Bigr)(1-\chi_{\{0\}}(x_j(t)-x_{k+1}(t)))v_{k+1}K'(x_{k+1}(t)-x_j(t))\notag\\
   	&-\Bigl(\frac{N}{c_L}M_{\rho^N_L(t)}(x)-k\Bigr)(1-\chi_{\{0\}}(x_j(t)-x_k(t)))v_kK'(x_{k}(t)-x_j(t))\Bigr].\label{eq:Klin}
   \end{align}
and \begin{align}
	\phi_v^{\mathrm{lin}}(\rho^N_L(t,x))&=\sum_k\chi_{[x_k(t),x_{k+1}(t))}(x)\Bigl[\bigl(\phi_v(\rho_k)-\phi_v(\rho_{k-1})\bigr)\notag\\
	&+\Bigl(\frac{ N}{c_L}M_{\rho^N_L(t)}(x)-k\Bigr)\bigl( (\phi_v(\rho_{k+1})-\phi_v(\rho_{k})) - (\phi_v(\rho_{k})-\phi_v(\rho_{k -1}))\bigr) \Bigr]\Big).
\end{align}

One has the following proposition, which extends Proposition 2.1 in \cite{DRR} to the deterministic particle approximation for aggregation-diffusion equations with nonlinear mobility.
\begin{proposition}\label{prop:approxpde} Let $T>0$ be such that $\rho^N_L$ is well-defined on $[0,T)$. Then, $\rho^N_L$ is a weak solution of the following PDE on $[0,T)\times \T_L$
	\begin{align}\label{eq:pdeapprox}
		\partial_t\rho^N_L = & \partial_x\Big(\rho^N_L {K^\lin}'(\rho^N_L, v^N_L)\ast\rho^N_L+\frac{ N}{c_L}\rho^N_L\phi_v^{\mathrm{lin}}(\rho^N_L)\Big).
	\end{align}

The PDE \eqref{eq:pdeapprox} is satisfied in the following sense: for all $\varphi\in Lip([0,T)\times\T_L)\cap C_0([0,T)\times\T{_L})$ it  holds 
	\begin{align}
		\int_0^T\int_{\T_L}\rho^N_L(t,x)\partial_t\varphi(t,x)&\dx\dt=-\int_{\T_L}\varphi(0,x)\rho^N_L(0,x)\dx\notag\\
		&-\int_0^T\int_{\T^L}\partial_x\varphi(t,x)\rho^N_L(t,x){K^\lin}'(\rho^N_L, v^N_L)(t,\cdot)\ast\rho^N_L(t,x)\dx\dt\notag\\
		&-\int_0^T\int_{\T_L}\partial_x\varphi(t,x)\frac{ N}{c_L}\rho^N_L(t,x)	\phi_v^{\mathrm{lin}}(\rho^N_L(t,x))\dx\dt.
	\end{align} 
\end{proposition}

The proof of Proposition \ref{prop:approxpde} is similar to the proof of Proposition 2.1 in \cite{DRR}. It exploits the fact that, thanks also to the continuity of the mobility function $v$, for any fixed $z\in[0,c_L]$ the function $t\mapsto X(t,z)$ is a $C^1$ function on the time interval on which $\rho^N_L$ is well defined. We report the proof for completeness. 

\begin{proof}
	Assume that $\rho^N_L$ is well-defined on the interval $[0,T)$. In particular, $x_0(t)<\dots<x_{N-1}(t)$ for all $t\in[0,T)$. Since $K\in C^1(\R^d\setminus \{0\})$, $\phi\in C^0$ and $v\in C^0$, then $x_k$ and $\rho_k$ are $C^1$ functions of time. In particular, the function 
	\begin{equation}\label{eq:xtz}
		X(t,z)=\sum_{k=0}^{N-1}\chi_{[kc_L/N,(k+1)c_L/N)}(z)x_k(t)+\Big(\frac{N}{c_L}z-k\Big)(x_{k+1}(t)-x_k(t))
	\end{equation}   
	is $C^1$ as  function of time for any fixed $z\in[0,c_L]$. 
	Using the change of variables $x=X(t,z)$ and the fact that the function $t\mapsto X(t,z)$ is $C^1$ for any fixed $z\in[0,c_L]$ one has that
	\begin{align}
		\int_0^T\int_{\T_L}\rho^N_L(t,x)\partial_t\varphi(t,x)\dx\dt&=\int_0^T\int_0^{c_L}\partial_t\varphi(t,X(t,z))\dz\dt\\
		&=-\int_{\T_L}\varphi(0,x)\rho^N_L(0,x)\dx\\
		&-\int_0^T\int_0^{c_L}\partial_x\varphi(t,X(t,z))\partial_tX(t,z)\dz\dt.\label{eq:2.6}
	\end{align} 
	Differentiating w.r.t. $t$ the function $X(t,z)$ in \eqref{eq:xtz} and inserting the formula for $\dot x_k(t)$ given in \eqref{eq:ode} one gets
	\begin{align}
		\partial_tX(t,z)&=\sum_{k=0}^{N-1}\chi_{[kc_L/N,(k+1)c_L/N)}(z)\dot x_k(t)+\Big(\frac{N}{c_L}z-k\Big)(\dot x_{k+1}(t)-\dot x_k(t))\notag\\
		&=\sum_{k=0}^{N-1}\chi_{[kc_L/N,(k+1)c_L/N)}(z)\Bigl[-\frac{c_L}{N}\sum_{j\neq k}v_kK'(x_k-x_j)\notag\\
		&-\frac{c_L}{N}\Bigl(\frac{N}{c_L}z-k\Bigr)\Bigl(\sum_{i \neq k+1}v_{k+1}K'(x_{k+1}-x_j)-\sum_{i \neq k}v_kK'(x_{k}-x_j)\Bigr)\Bigr]\notag\\
		&-\frac{N}{c_L}[\phi_v(\rho_k)-\phi_v(\rho_{k-1})]-\frac{N}{c_L}\Bigl(\frac{N}{c_L}z-k\Bigr)[(\phi_v(\rho_{k+1})-\phi_v(\rho_k))+(\phi_v(\rho_{k-1})-\phi_v(\rho_k))].\label{eq:dtx}
	\end{align}
	Inserting the formula above for $\partial_tX(t,z)$ in \eqref{eq:2.6} and using the change of variables $z=M_{\rho^N_L(t)}(x)$ one has that
	\begin{align}
		\int_0^T\int_{\T_L}\rho^N_L(t,x)\partial_t\varphi(t,x)&\dx\dt=-\int_{\T_L}\varphi(0,x)\rho^N_L(0,x)\dx\notag\\
		&-\int_0^T\int_{\T_L}\partial_x\varphi(t,x)\sum_{k=0}^{N-1}\chi_{[kc_L/N,(k+1)c_L/N)}(M_{\rho^N_L(t)}(x))\Bigl[-\frac{c_L}{N}\sum_{j\neq k}v_kK'(x_k-x_j)\notag\\
		&-\frac{c_L}{N}\Bigl(\frac{N}{c_L}M_{\rho^N_L(t)}(x)-k\Bigr)\Bigl(\sum_{i \neq k+1}v_{k+1}K'(x_{k+1}-x_j)-\sum_{i \neq k}v_kK'(x_{k}-x_j)\Bigr)\Bigr]\rho^N_L(t,x)\dx\dt\notag\\
		&-\int_0^T\int_{\T_L}\partial_x\varphi(t,x)\frac{ N}{c_L}\rho^N_L(t,x)\sum_k\chi_{[x_k(t),x_{k+1}(t))}(x)\Bigl[\bigl(\phi_v(\rho_k)-\phi_v(\rho_{k-1})\bigr)\notag\\
		&+\Bigl(\frac{ N}{c_L}M_{\rho^N_L(t)}(x)-k\Bigr)\bigl( (\phi_v(\rho_{k+1})-\phi_v(\rho_{k})) - (\phi_v(\rho_{k})-\phi_v(\rho_{k -1}))\bigr) \Bigr]\dx\dt\notag\\
		&=-\int_{\T_L}\varphi(0,x)\rho^N_L(0,x)\dx\notag\\
		&-\int_0^T\int_{\T^L}\partial_x\varphi(t,x)\rho^N_L(t,x){K^\lin}'(\rho^N_L, v^N_L)(t,\cdot)\ast\rho^N_L(t,x)\dx\dt\notag\\
		&-\int_0^T\int_{\T_L}\partial_x\varphi(t,x)\frac{ N}{c_L}\rho^N_L(t,x)\phi_v^{\mathrm{lin}}(\rho^N_L(t,x))\dx\dt,\notag
	\end{align} 
	thus proving \eqref{eq:pdeapprox}.
	
	Thus the piecewise constant approximations $\rho^N_L$  satisfy on $\T_L$ the PDE \eqref{eq:pdeapprox} in the weak sense.
\end{proof}

Due to the translation invariance of the torus, we can assume w.l.o.g. that $x_0(t)=x_N(t)$ is fixed during the evolution.

  \subsection{Bounds on the nonlocal interaction term}
  In this section we collect some useful estimates relating the nonlocal term of the aggregation diffusion equation \eqref{eq:mainPDE} with the corresponding term of the approximate PDE \eqref{eq:pdeapprox}.  The estimates below will be used when estimating the time derivative of the energy functional $\Fcal^v_L$ along the deterministic particle approximation $\rho^N_L$ (Lemma \ref{lemma:decrease}).
  
   We recall that $K''_{|_{\R\setminus\{0\}}} \in L^{\infty}(\R \setminus \{0\})$ and  with a slight abuse of notation we denote from now on the uniform bound of $K''$ on $\R\setminus\{0\}$ with $\| K'' \|_{L^\infty}$. 
   \begin{lemma}
  	\label{lemma:kbounds}
  	Let $K$ be as in \eqref{eq:k1}-\eqref{eq:k3} and let $\rho^N_L$ be defined as in \eqref{eq:rhonl}. The following holds:
  	\begin{equation}\label{eq:kerbound}
  		| K' \ast \rho^N_L (t, X(t,z))  | \leq c_L\|K' \|_{L^\infty}, \qquad\forall\,z\in[0,c_L),
  	\end{equation} 
  	and if $z \in [kc_L/N, (k+1)c_L/N)$, then 
  	\begin{align}
  		\Bigl|\int_0^{c_L} K'\ast \rho^N_L (t,X(t,z)) &\big[ v^N_L(t,X(t,z))K' \ast \rho^N_L (t,X(t,z)) -  {K^\lin}'(\rho^N_L, v^N_L) \ast \rho^N_L (t,X(t,z)) \big] \dz\Bigr|\leq\notag\\
  	    &\leq {C \|K' \|_{L^\infty} \|v \|_{\infty} \big(  \|K'' \|_{L^\infty}L + \|K' \|_{L^\infty} \big) \frac{c_L^2}{N}.} \label{eq:kerLip}
  		%&\leq 2\|v\|_{L^\infty}\| K'' \|_{L^\infty} |x_{k+1} - x_k| + c_L\frac{L\|v\|_{L^\infty}\| K'' \|_{L^\infty} + 3 \| K' \|_{L^\infty}}{N}.
  	\end{align}
  \end{lemma}
  
  \begin{proof}
  	The bound \eqref{eq:kerbound} follows directly from the fact that  $\int_{\T_L}\rho^N_L=c_L$.
  	
  	To prove the inequality \eqref{eq:kerLip}, we use the definition of ${K^\lin}'(\rho^N_L,v^N_L)$ in  \eqref{eq:Klin} as follows 
  	\begin{align*}
  		\Bigl|&\int_0^{c_L} K'\ast \rho^N_L (t,X(t,z)) \big[ v^N_L (t,X(t,z)) K' \ast \rho^N_L (t,X(t,z)) -  {K^\lin}'(\rho^N_L,v^N_L) \ast \rho^N_L (t,X(t,z)) \big] \dz\Bigl|= \\
  		&=\Bigl|  \sum_{k=0}^{N-1} \int_{x_k}^{x_{k+1}} K' \ast \rho^N_L(t,x) v_k(t) \rho_k(t) \sum_{j\neq k} \int_{x_j}^{x_{j+1}} \big[K'(x_k(t) - x_j(t)) - K'(x-y) \big] \rho_j(t)  \dy \dx  \\
  		& -  \sum_{k=0}^{N-1} \int_{x_k}^{x_{k+1}} K' \ast \rho^N_L(t,x) v_k(t) \rho_k(t)  \int_{x_k}^{x_{k+1}} K'(x-y) \rho_k(t) \dy \dx \\
  		& + \sum_{k=0}^{N-1} \int_{kc_L/N}^{(k+1)c_L/N} K' \ast \rho^N_L(t, X(t,z)) \left(\frac{N}{c_L}z - k \right) \frac{c_L}{N} \Bigl( v_{k+1}(t) \sum_{j\neq k+1} K'(x_{k+1} - x_j) - v_k(t) \sum_{j\neq k} K'(x_k - x_j)  \Bigr) \dz\Bigr|.
  		\end{align*}  
  	Then we proceed to estimate the different terms in the above formula. First one has that
  	\begin{align*}
  		\Bigl|&  \sum_{k=0}^{N-1} \int_{x_k}^{x_{k+1}} K' \ast \rho^N_L(t,x) v_k(t) \rho_k(t) \sum_{j\neq k} \int_{x_j}^{x_{j+1}} \big[K'(x_k(t) - x_j(t)) - K'(x-y) \big] \rho_j(t)  \dy \dx\Bigr|\leq\\
  	 &\leq \|K' \|_{L^\infty}  \|K'' \|_{L^\infty}  \sum_{k=0}^{N-1} \int_{x_k}^{x_{k+1}} v_k(t) \rho_k(t) \left(\frac{c_L}{N}\sum_{j \neq k} \max\bigl\{|x_{k+1}(t) - x_k(t)|,|x_j(t)-x_{j+1}(t)|\bigr\}\right) \dx\\
  	 &\leq 2c^2_L\|K' \|_{L^\infty}\|K''\|_{L^\infty}\|v\|_{L^\infty}\frac{L}{N},
  	 \end{align*}
   where we used the fact that $\sum_{k=0}^{N-1}|x_k-x_{k+1}|=L$.
   
   The second term can be estimated as follows
  \begin{align*}
  	\Bigg|&\sum_{k=0}^{N-1} \int_{x_k}^{x_{k+1}} K' \ast \rho^N_L(t,x) v_k(t) \rho_k(t)  \int_{x_k}^{x_{k+1}} K'(x-y) \rho_k(t) \dy \dx\Bigg|\leq\\
  	 &\leq  \|K' \|^2_{L^\infty} \| v\|_{\infty}  \frac{c_L^2}{N}.
  	\end{align*}
  We conclude with the following estimate
  \begin{align*}
  	\Bigg|&\sum_{k=0}^{N-1} \int_{kc_L/N}^{(k+1)c_L/N} K' \ast \rho^N_L(t, X(t,z)) \left(\frac{N}{c_L}z - k \right) \frac{c_L}{N} \Bigl( v_{k+1}(t) \sum_{j\neq k+1} K'(x_{k+1} - x_j) - v_k(t) \sum_{j\neq k} K'(x_k - x_j)  \Bigr) \dz\Bigg|\leq\\
  		&\leq\Bigg|\sum_{k=0}^{N-1} \Bigg[ \int_{(k-1)c_L/N}^{kc_L/N} v_k K' \ast \rho^N_L(t,X(t,z)) \left(\frac{N}{c_L}z - k +1\right) \dz  \\
  		&- \int^{(k+1)c_L/N}_{kc_L/N} v_k K' \ast \rho^N_L(t,X(t,z)) \left(\frac{N}{c_L}z - k\right) \dz \Bigg] \frac{c_L}{N} \sum_{j\neq k} K'(x_k - x_j)\Bigg| \\
  		&\leq\Bigg| \sum_{k=0}^{N-1} \int^{(k+1)c_L/N}_{kc_L/N}  v_k  \left(\frac{N}{c_L}z - k\right) \big( K' \ast \rho^N_L(t,X(t,z - c_L/N)) - K' \ast \rho^N_L(t,X(t,z))  \big)  \frac{c_L}{N} \sum_{j\neq k} K'(x_k - x_j) \dz \Bigg|\\
  		&\leq  C \|K' \|_{L^\infty} \|v \|_{\infty}\|K'' \|_{L^\infty} \frac{c_L^2 L }{N}. 
  	\end{align*} 
  Thus \eqref{eq:kerLip} is proved.

  \end{proof}

 \section{Estimates on the energy of the particle approximations}
 
 The aim of this section is to provide explicit computations and estimates on the energy of the discrete particle approximations. 
  
   In this section we will sometimes denote with $\rho^N_L(t)$ the function $\rho^N_L(t,\cdot):\T_L\to[0,+\infty)$. 
   
  We have the following estimate on the time derivative of the energy $\Fcal^v_L$ along the discrete particle approximations $\rho^N_L$, where
    \begin{equation}
  \label{eq:FcalL}
  \begin{split}
  \Fcal^v_{L}(\rho) :=\frac12\int_{\T_L}\int_{\T_L} K(x-y) \rho(x)\rho(y)\dx\dy +\int_{\T_L}W_v(\rho)\dx,
  \end{split}
  \end{equation}
and $W_v:\R\to\R$ is such that
\begin{equation}
	\phi_v(s)=sW_v'(s)-W_v(s).
\end{equation}
In particular, $W_v''(s)=\phi_v'(s)\geq0$, thus $W_v$ is a convex function. 
  
\begin{lemma}\label{lemma:decrease}
	Assume $\rho^N_L(t)$ is well defined on $[0,T)$ and $N\geq \bar N$ for some sufficiently large $\bar N$ depending on $\|K'\|_{L^\infty},\|K''\|_{L^\infty}$ and $\|v\|_{L^\infty}$. Then, for all $t\in[0,T)$ the functional $\Fcal^v_L(\rho^N_L(t))$ satisfies
	\begin{align}
	\frac{d}{dt}\Fcal^v_L(\rho^N_L(t))&\leq -\Bigl(1-C_1\frac{L}{\sqrt{N}}\Bigr) \sum_{k=0}^{N-1} \int_{kc_L/N}^{(k+1)c_L/N}  \frac{N^2}{c_L^2} \bigl(\phi_v(\rho_k(t))-\phi_v(\rho_{k-1}(t))\bigr)^2\dz\notag\\
		&+C_2 \Biggl(\sum_{k=0}^{N-1} \int_{kc_L/N}^{(k+1)c_L/N}  \frac{N^2}{c_L^2} \bigl(\phi_v(\rho_k(t))-\phi_v(\rho_{k-1}(t))\bigr)^2\dz\Biggr)^{1/2} +C_1 \frac{L}{\sqrt{N}}+C_2,\label{eq:gradient}
	\end{align}
where $C_1=C_1\bigl(\|K'\|_{L^\infty}, \|K''\|_{L^\infty}, \|v\|_{L^\infty})$ and $C_2=C_2\bigl(\|K'\|_{L^\infty},\|v\|_{L^\infty}\bigr)$.
\end{lemma}

%Estimate~\eqref{eq:gradient} is natural since the discrete particle approximations $\rho^N_L$ satisfy the PDE~\eqref{eq:pdeapprox}, which is an approximate version of the gradient flow in the generalized Wasserstein sense of the functional $\Fcal_L$.

From Lemma~\ref{lemma:decrease}, we have the following discrete $W^{1,2}$ bound in the pseudoinverse variables for the function $\phi_v(\rho^N_L)$, which will be crucial to obtain strong $L^1$-compactness of the approximations. 
\begin{corollary}\label{cor:energybound}
 Let $T>0$, $\eps>0$,  $L>0$, $0<C_3<1-\eps$ and $C_2$ the constant appearing in \eqref{eq:gradient}. Then there exists $\bar N=\bar N(T,\eps,L,\|K'\|_{L^\infty}, \|K''\|_{L^\infty}, \|v\|_{L^\infty})$ and $C'=C'(\eps, C_2,C_3,\|K'\|_{L^\infty})$ such that for all $N\geq \bar N$ if $\rho^N_L$ are the discrete particle approximations with initial datum $\rho_{0,L}$ and they are well defined on $[0,T)$, then 
 \begin{align}\label{eq:stimaw12}
 	\int_0^T\sum_{k=0}^{N-1}\int_{kc_L/N}^{(k+1)c_L/N}\frac{N^2}{c_L^2}\bigl(\phi_v(\rho_{k}(t))-\phi_v(\rho_{k-1}(t))\bigr)^2\dz\dt\leq C'(1+T)+\frac{1}{C_3}\Fcal^v_L(\rho_{0,L}).
 \end{align}
\end{corollary}

Let us first show how Corollary \ref{cor:energybound} follows from Lemma \ref{lemma:decrease}.

\begin{proof}
	[Proof of Corollary \ref{cor:energybound}:]

For any $\eps>0$, there exists $\bar N=\bar N(T,\eps,L,\|K'\|_{L^\infty}, \|K''\|_{L^\infty}, \|v\|_{L^\infty})$ such that for any $N\geq\bar N$ it holds 
\[
C_1\frac{L}{\sqrt{N}}\leq \eps,
\]
where $C_1=C_1\bigl(\|K'\|_{L^\infty}, \|K''\|_{L^\infty}, \|v\|_{L^\infty})$ is the constant appearing in Lemma \ref{lemma:decrease}. 

Setting for simplicity 
\begin{equation}\label{eq:a}
a^2(t)=\sum_{k=0}^{N-1}\int_{kc_L/N}^{(k+1)c_L/N}\frac{N^2}{c_L^2}\bigl(\phi_v(\rho_{k}(t))-\phi_v(\rho_{k-1}(t))\bigr)^2\dz,
\end{equation} 
the inequality \eqref{eq:gradient} when $N\geq\bar N$ reads as
\begin{align}\label{eq:gradient_a1}
\frac{d}{dt}\Fcal^v_L(\rho^N_L(t))&\leq-(1-\eps)a^2(t)+C_2a(t)+C_2+\eps.	
\end{align}
Now notice from the above that there exists $\bar C=\bar C(\eps,C_2,C_3)$ such that whenever $a^2(t)\geq\bar C$ one has that
\begin{equation}\label{eq:gradient_a}
\frac{d}{dt}\Fcal^v_L(\rho^N_L(t))\leq-C_3a^2(t)+C_2+\eps.	
\end{equation}
Since $t\mapsto a(t)$ is a continuous function, the set
\[
\Omega(\bar C)=\{t\in[0,T]:\,a^2(t)<\bar C\}=\underset{i\in\N}\bigcup(t_{2i},t_{2i+1})
\]
is open. 

For $(t_{2i},t_{2i+1})\subset\Omega(\bar C)$ by definition it holds
\begin{align}
	\int_{t_{2i}}^{t_{2i+1}}a^2(t)\dt\leq\bar C(t_{2i+1}-t_{2i}), 
\end{align}
while for  $(t_{2i-1},t_{2i})\subset\Omega(\bar C)^c$, by \eqref{eq:gradient_a} it holds 
\begin{align}
	\int_{t_{2i-1}}^{t_{2i}}a^2(t)\dt&\leq \frac{1}{C_3}(C_2+\eps)(t_{2i}-t_{2i-1})+\frac{1}{C_3}\bigl(\Fcal^v_L(\rho^N_L(t_{2i-1}))-\Fcal^v_L(\rho^N_L(t_{2i}))\bigr)\notag\\
	&\leq\frac{1}{C_3}(C_2+\eps)(t_{2i}-t_{2i-1})+\frac{1}{C_3}\bigl(\Fcal^v_L(\rho^N_L(t_{2i-1}))-\min\Fcal^v_L\bigr)\notag\\
		&\leq\frac{1}{C_3}(C_2+\eps)(t_{2i}-t_{2i-1})+\frac{1}{C_3}\bigl(\Fcal^v_L(\rho^N_L(t_{2i-1}))+\|K'\|_{L^\infty}\bigr)\label{eq:contoa}
\end{align}
Now notice that by  \eqref{eq:gradient_a1} 
\begin{equation}
	\Fcal^v_L(\rho^N_L(t_{2i-1}))\leq\Fcal^v_L(\rho^N_L(t_{2i-2}))+ (C_2\sqrt{\bar C}+C_2+\eps)(t_{2i-1}-t_{2i-2})
\end{equation}
and that by \eqref{eq:gradient_a}
\begin{equation}
\Fcal^v_L(\rho^N_L(t_{2i-2}))\leq 	\Fcal^v_L(\rho^N_L(t_{2i-3})) +(C_2+\eps)(t_{2i-2}-t_{2i-3}).	
\end{equation}
Finally, from \eqref{eq:contoa} and iterating the  estimates \eqref{eq:gradient_a1} and \eqref{eq:gradient_a} up to $i=0$, recalling \eqref{eq:a} we obtain
\begin{align}
	\int_{0}^T\sum_{k=0}^{N-1}\int_{kc_L/N}^{(k+1)c_L/N}\frac{N^2}{c_L^2}\bigl(\phi_v(\rho_{k}(t))-\phi_v(\rho_{k-1}(t))\bigr)^2\dz\dt&\leq \Bigl(\bar C+\frac{1}{C_3}(C_2+\eps)\Bigr)T\notag\\
	&+\frac{1}{C_3}(C_2\sqrt{\bar C}+C_2+\eps)T\notag\\
	&+\frac{1}{C_3}\bigl(\Fcal^v_L(\rho_{0,L})+\|K'\|_{L^\infty}\bigr)\notag\\
	&\leq C'(1+T)+\frac{1}{C_3}\Fcal^v_L(\rho_{0,L}).
\end{align}

\end{proof}

\begin{proof}
	[Proof of Lemma~\ref{lemma:decrease}:]  
	Let us recall the notation 
	\begin{align}
			G_k(t)&:= \phi_v(\rho_k(t)) -  \phi_v(\rho_{k-1}(t)). 	\label{eq:Fk}
	\end{align}

\textbf{Step 1 (Computation of the derivative)}

	Via the changes of variables $x=X(t,z)$ and $y=X(t,\xi)$ in the second term of $\Fcal_L(\rho^N_L(t))$ one has that 
	\begin{align}
		\frac{d}{dt}\Fcal^v_{L}(\rho^N_L(t)) &=  \frac{d}{dt}\int_{\T_L}   W_v(\rho^N_L(t,x)) \dx +  \frac{d}{dt} \frac12\int_{0}^{c_L} K\ast\rho^N_L(t, X(t,z))\dz\notag\\
		& =  \frac{d}{dt}\int_{\T_L}   W_v(\rho^N_L(t,x)) \dx +  \frac{d}{dt} \frac12\int_{0}^{c_L}\int_0^{c_L} K(X(t,z)-X(t,\xi))\d\xi\dz.\label{eq:derformv}
	\end{align}
	Since $K$ is Lipschitz and $t\mapsto X(t,z)$ belongs to $C^1([0,T))$ as observed through the explicit formula \eqref{eq:xtz} in the proof of Proposition \ref{prop:approxpde}, then one can differentiate the second term in the r.h.s. of \eqref{eq:derformv} and obtain  
	\begin{align}
		\frac{d}{dt} \frac12\int_{0}^{c_L}&\int_0^{c_L} K(X(t,z)-X(t,\xi))\d\xi\dz=\frac12\int_{0}^{c_L}\int_0^{c_L}K'(X(t,z)-X(t,\xi))\Bigl(\partial_tX(t,z)-\partial_tX(t,\xi)\Bigr)\d\xi\dz\notag\\
		&=\int_{0}^{c_L}\int_0^{c_L}K'(X(t,z)-X(t,\xi))\partial_tX(t,z)\d\xi\dz\notag\\
		&=-\int_0^{c_L}K'\ast\rho^N_L(t,X(t,z)){K^\lin}'(\rho^N_L, v^N_L)\ast\rho^N_L(t,X(t,z))\dz\notag\\
		& - \frac{N}{c_L} \sum_{k=0}^{N-1} \int_{kc_L/N}^{(k+1)c_L/N} K'\ast \rho^N_L (t,X(t,z))\Bigl[ G_k(t) + \Bigl(\frac{N}{c_L}z - k\Bigr)(G_{k+1}(t) - G_k(t)) \Bigr] \dz\notag\\ 
		&= - \int_0^{c_L} \left( K' \ast \rho^N_L (t,X(t,z)) \right)^2 v^N_L(t,X(t,z))\dz \notag\\
		& + \int_0^{c_L} K'\ast \rho^N_L (t,X(t,z)) \big[K' \ast \rho^N_L (t,X(t,z))v^N_L(t,X(t,z)) -  {K^\lin}'(\rho^N_L,v^N_L)\ast \rho^N_L (t,X(t,z)) \big] \dz \notag\\
		& - \frac{N}{c_L} \sum_{k=0}^{N-1} \int_{kc_L/N}^{(k+1)c_L/N} K'\ast \rho^N_L (t,X(t,z))\Bigl[ G_k(t)+ \Bigl(\frac{N}{c_L}z - k\Bigr)(G_{k+1}(t) - G_k(t)) \Bigr] \dz. \notag
	\end{align}
	where we used the explicit formula \eqref{eq:dtx} as in the proof of Proposition \ref{prop:approxpde}.
	
	%	\begin{align}\label{eq:derform2}
	%	\frac{d}{dt}\Fcal_{L}(\rho^N_L(t)) &=  \frac{d}{dt}\int_0^c_L\int_{\T_L}   W(\rho^N_L(t,x)) \dx +  \frac{d}{dt}\int_{\T_L} K\ast\frac{d}{dt}\rho^N_L(t, X(t,z))\dz\notag\\
	%	&=\frac{d}{dt}\int_{\T_L}   W(\rho^N_L(t,x)) \dx -  \int_{\T_L} K'\ast\rho^N_L(t, X(t,z))\partial_tX(t,z)\dz
	%\end{align}

	%	Using~\eqref{eq:pdeapprox}, integration by parts and applying the pseudoinverse change of variables, one can rewrite the second term of the r.h.s. of~\eqref{eq:derform} as follows

	On the other hand, observing that the ODE~\eqref{eq:ode} can be rewritten as
	\begin{equation*}
		\dot x_k(t) =-\int_{\T_L} {K^\lin}'(\rho^N_L, v^N_L)(x_k-y)\rho^N_L(t,y)\dy-\frac{N}{c_L}G_k(t),
	\end{equation*}
	the first term of the r.h.s. of~\eqref{eq:derformv} can be explicitly computed obtaining
	\begin{align*}
		\frac{d}{dt}\int_{\T_L}&  W_v(\rho^N_L(t,x)) \dx =\frac{d}{dt}\sum_{k=0}^{N-1}(x_{k+1}(t)-x_k(t)) W_v(\rho_k(t))\\
		&=\sum_{k=0}^{N-1}(\dot{x}_{k+1}(t)-\dot{x}_k(t))\phi_v(\rho_k(t))\\
		&=\sum_{k=0}^{N-1}\dot{x}_k(t)G_k(t)\\
		&= -\frac{N}{c_L} \sum_{k=0}^{N-1}  G^2_k(t) -  \sum_{k=0}^{N-1}  G_k(t)  {K^\lin}'(\rho^N_L,v^N_L)\ast \rho^N_L (t,x_k)   \\
		&= -\frac{N^2}{c_L^2}  \sum_{k=0}^{N-1} \int_{kc_L/N}^{(k+1)c_L/N}  G^2_k(t)\dz -\frac{N}{c_L} \sum_{k=0}^{N-1} \int_{kc_L/N}^{(k+1)c_L/N}  G_k(t) K' \ast \rho^N_L (t,X(t,z))v^N_L(t,X(t,z)) \dz\\
		& - \frac{N}{c_L} \sum_{k=0}^{N-1} \int_{kc_L/N}^{(k+1)c_L/N}  G_k(t) [ K' \ast \rho^N_L (t,X(t,z))v^N_L(t,X(t,z))  -  {K^\lin}'(\rho^N_L,v^N_L) \ast \rho^N_L (t,x_k) ] \dz.
	\end{align*}

%	Once here, we observe that the periodicity of the torus ensures that $\sum_kF_k^2=\sum_kF_{k+1}^2$ and as a consequence by standard computations
%	\[\sum_{k=0}^{N-1} \int_{kc_L/N}^{(k+1)c_L/N} (F_k(t))^2 \dz \geq \sum_{k=0}^{N-1} \int_{kc_L/N}^{(k+1)c_L/N}  \Bigl(F_k(t) + \Bigl(\frac{N}{c_L}z-k\Bigr)(F_{k+1}(t) - F_k(t)) \Bigr)^2 \dz.   \]
	Then,
	\begin{align}\label{eq:ender}
		\notag
		\frac{d}{dt}\Fcal^v_{L}&(\rho^N_L(t)) \leq  - \sum_{k=0}^{N-1} \int_{kc_L/N}^{(k+1)c_L/N}  \frac{N^2}{c_L^2} G^2_k(t) + \bigl(K'\ast \rho^N_L (t,X(t,z))\bigr)^2v^N_L(t,X(t,z)) \dz \notag\\
		&-\sum_{k=0}^{N-1}\int_{kc_L/N}^{(k+1)c_L/N} \frac{N}{c_L}G_k(t)K'\ast \rho^N_L (t,X(t,z))(1+v^N_L(t,X(t,z)))\dz\notag\\
		& - \sum_{k=0}^{N-1} \int_{kc_L/N}^{(k+1)c_L/N} K'\ast \rho^N_L (t,X(t,z)) \frac{N}{c_L}\Bigl(\frac{N}{c_L}z - k\Bigr)(G_{k+1}(t) - G_k(t)v_k(t)) \dz \notag\\
		\notag
		& -  \sum_{k=0}^{N-1} \int_{kc_L/N}^{(k+1)c_L/N} \frac{N}{c_L} G_k(t) [ K' \ast \rho^N_L (t,X(t,z))v^N_L(t,X(t,z))  -  {K^\lin}'(\rho^N_L,v^N_L) \ast \rho^N_L (t,x_k) ]\dz \\
		&- \sum_{k=0}^{N-1} \int_{kc_L/N}^{(k+1)c_L/N} K' \ast \rho^N_L (t,X(t,z)) \big[K'\ast \rho^N_L (t,X(t,z)) v^N_L(t,X(t,z))-  {K^\lin}'(\rho^N_L,v^N_L) \ast \rho^N_L (t,X(t,z)) \big] \dz\notag\\
		&=A_1+A_2+A_3+A_4+A_5,
	\end{align}
where
\begin{align*}
	A_1&:=   - \sum_{k=0}^{N-1} \int_{kc_L/N}^{(k+1)c_L/N}  \frac{N^2}{c_L^2} G^2_k(t) + \bigl(K'\ast \rho^N_L (t,X(t,z))\bigr)^2v^N_L(t,X(t,z)) \dz  \notag\\
	A_2&=-\sum_{k=0}^{N-1}\int_{kc_L/N}^{(k+1)c_L/N} \frac{N}{c_L}G_k(t)K'\ast \rho^N_L (t,X(t,z))(1+v^N_L(t,X(t,z)))\dz\notag\\
	A_3&:=- \sum_{k=0}^{N-1} \int_{kc_L/N}^{(k+1)c_L/N} K'\ast \rho^N_L (t,X(t,z)) \frac{N}{c_L}\Bigl(\frac{N}{c_L}z - k\Bigr)(G_{k+1}(t)- G_k(t)) \dz \notag\\
	A_4&:=-  \sum_{k=0}^{N-1} \int_{kc_L/N}^{(k+1)c_L/N} \frac{N}{c_L} G_k(t) [ K' \ast \rho^N_L (t,X(t,z))v^N_L(t,X(t,z))  -  {K^\lin}'(\rho^N_L,v^N_L) \ast \rho^N_L (t,x_k) ]\dz \notag\\
A_5&:=- \sum_{k=0}^{N-1} \int_{kc_L/N}^{(k+1)c_L/N} K' \ast \rho^N_L (t,X(t,z)) \big[K'\ast \rho^N_L (t,X(t,z))v^N_L(t,X(t,z)) -  {K^\lin}'(\rho^N_L,v^N_L) \ast \rho^N_L (t,X(t,z)) \big] \dz.
\end{align*}
Let us estimate all the above terms.

\textbf{Step 2 (Estimate of the term $A_5$)}
 The last term in \eqref{eq:ender} can be estimated by Lemma \ref{lemma:kbounds} as follows
	\begin{align}
|A_5|\leq{C \|K' \|_{L^\infty} \|v \|_{\infty} \big(  \|K'' \|_{L^\infty}L + \|K' \|_{L^\infty} \big) \frac{c_L^2}{N}.} \label{eq:A5}
\end{align}
\textbf{Step 3 (Estimate of the term $A_4$)}
 Let us now consider the term $A_4$. Similarly to the proof of Lemma \ref{lemma:kbounds}, using the definitions of the various quantities, one has that 
 \begin{align}
 	|&A_4|\leq 	\Bigl| \sum_{k=0}^{N-1} \int_{kc_L/N}^{(k+1)c_L/N} \frac{N}{c_L}G_k(t) v_k(t) \sum_{j\neq k} \int_{x_j}^{x_{j+1}} \big[K'(x_k(t) - x_j(t)) - K'(X(t,z)-y) \big] \rho_j(t)  \dy \dz\Bigr|\notag\\
 	&+\Bigg|\sum_{k=0}^{N-1} \int_{kc_L/N}^{(k+1)c_L/N} \frac{N}{c_L}G_k(t) v_k(t)  \int_{x_k}^{x_{k+1}} K'(X(t,z))-y) \rho_k(t) \dy \dz\Bigg|\notag\\
 	&+\Bigg|\sum_{k=0}^{N-1} \int_{kc_L/N}^{(k+1)c_L/N} \frac{N}{c_L}G_k(t) \left(\frac{N}{c_L}z - k \right) \frac{c_L}{N} \Bigl( v_{k+1}(t) \sum_{j\neq k+1} K'(x_{k+1} - x_j) - v_k(t) \sum_{j\neq k} K'(x_k - x_j)  \Bigr) \dz\Bigg|.\label{eq:a3est}
 \end{align}
Then observe that, by H\"older inequality and the estimates used in Lemma \ref{lemma:kbounds}
\begin{align}
	\Bigl| &\sum_{k=0}^{N-1} \int_{kc_L/N}^{(k+1)c_L/N} \frac{N}{c_L}G_k(t) v_k(t) \sum_{j\neq k} \int_{x_j}^{x_{j+1}} \big[K'(x_k(t) - x_j(t)) - K'(X(t,z)-y) \big] \rho_j(t)  \dy \dz\Bigr|\leq\notag\\
	&\leq\Bigg(\sum_{k=0}^{N-1}\int_{kc_L/N}^{(k+1)c_L/N}\Bigl|\frac{N}{c_L}G_k(t)\Bigr|^2v_k(t)\dz\Bigg)^{1/2}\cdot\notag\\
	&\cdot\Bigg(\sum_{k=0}^{N-1}\int_{kc_L/N}^{(k+1)c_L/N}\Bigg(\sum_{j \neq k}\int_{x_j}^{x_{j+1}}\big[K'(x_k(t) - x_j(t)) - K'(X(t,z)-y) \big] \rho_j(t) \dy\Bigg)^2v_k(t) dz\Bigg)^{1/2}\notag\\
	&\leq C\|K''\|_{L^\infty}\|v\|_{L^\infty}\Bigg(\sum_{k=0}^{N-1}\int_{kc_L/N}^{(k+1)c_L/N}\Bigl|\frac{N}{c_L}G_k(t)\dz\Bigr|^2\Bigg)^{1/2}\Bigl(\frac{L}{\sqrt{N}}+\frac{c_LL}{N}\Bigr).
\end{align}

Denoting for simplicity
\[  a_k = \frac{N}{c_L} G_k(t) , \qquad b_k = K' \ast \rho^N_L (t,X(t,z)) \chi_{[kc_L/N, (k+1)c_L/N)}(z),   \]
let us define
\begin{equation}
	\Omega_{k}=\bigl\{z\in[kc_L/N,(k+1)c_L/N]:\,|a_k|<2|b_k|\bigr\}.
\end{equation}

On the one hand, one has that 
\begin{align}
	\Bigl| &\sum_{k} \int_{[kc_L/N,(k+1)c_L/N]\cap \Omega_k} \frac{N}{c_L}G_k(t)v_k(t)  \sum_{j\neq k} \int_{x_j}^{x_{j+1}} \big[K'(x_k(t) - x_j(t)) - K'(X(t,z)-y) \big] \rho_j(t)  \dy \dz\Bigr|\leq\notag\\
	&\leq C\|K''\|_{L^\infty}\|K'\|_{L^\infty}\|v\|_{L^\infty}\frac{L}{\sqrt{N}}.
	\end{align}
On the other hand,
\begin{align}
	\Bigl| &\sum_{k} \int_{[kc_L/N,(k+1)c_L/N]\cap\omega_k^c} \frac{N}{c_L}G_k(t) v_k(t) \sum_{j\neq k} \int_{x_j}^{x_{j+1}} \big[K'(x_k(t) - x_j(t)) - K'(X(t,z)-y) \big] \rho_j(t)  \dy \dz\Bigr|\leq\notag\\
	&\leq C\|K''\|_{L^\infty}\|v\|_{L^\infty}\frac{L}{\sqrt{N}}\left( \sum_{k} \int_{[kc_L/N,(k+1)c_L/N]\cap\Omega_k^c}  |a_k|^2 \dz \right)^{1/2}.
	\end{align}
To estimate the above there are two cases. Either 
\[
\left(\sum_{k} \int_{[kc_L/N,(k+1)c_L/N]\cap\Omega_k^c}  |a_k|^2 \dz \right)^{1/2}\leq1
\]
and then 
\begin{align}
	\Bigl| \sum_{k} \int_{[kc_L/N,(k+1)c_L/N]\cap\Omega_k^c} \frac{N}{c_L}G_k(t) v_k(t) &\sum_{j\neq k} \int_{x_j}^{x_{j+1}} \big[K'(x_k(t) - x_j(t)) - K'(X(t,z)-y) \big] \rho_j(t)  \dy \dz\Bigr|\leq\notag\\
	&\leq C\|K''\|_{L^\infty}\|v\|_{L^\infty}\frac{L}{\sqrt{N}}
	\end{align}
or 
\[
\left(\sum_{k} \int_{[kc_L/N,(k+1)c_L/N]\cap\Omega_k^c}  |a_k|^2 \dz \right)^{1/2}>1
\]
and then 
\begin{align}
	\Bigl| &\sum_{k} \int_{[kc_L/N,(k+1)c_L/N]\cap\Omega_k^c} \frac{N}{c_L}G_k(t) v_k(t) \sum_{j\neq k} \int_{x_j}^{x_{j+1}} \big[K'(x_k(t) - x_j(t)) - K'(X(t,z)-y) \big] \rho_j(t)  \dy \dz\Bigr|\leq\notag\\
	&\leq C\|K''\|_{L^\infty}\|v\|_{L^\infty}\frac{L}{\sqrt{N}} \sum_{k} \int_{[kc_L/N,(k+1)c_L/N]\cap\Omega_k^c}  |a_k|^2 \dz\notag\\
		&\leq C\|K''\|_{L^\infty}\|v\|_{L^\infty}\frac{L}{\sqrt{N}} \sum_{k=0}^{N-1} \int_{kc_L/N}^{(k+1)c_L/N}  \frac{N^2}{c_L^2} G^2_k(t)\dz.\label{eq:a3est1}
	\end{align}
When $N$ is sufficiently large, this term will be then much smaller in absolute value than the first term in $A_1$.   . 

Analogously, the second term in \eqref{eq:a3est} can be estimated as follows
\begin{align}
	\Bigg|&\sum_{k=0}^{N-1} \int_{kc_L/N}^{(k+1)c_L/N} \frac{N}{c_L}G_k(t) v_k(t)  \int_{x_k}^{x_{k+1}} K'(X(t,x)-y) \rho_k(t) \dy \dz\Bigg|\leq\notag\\
	&\leq C\|K'\|_{L^\infty}\|v\|_{L^\infty}\frac{c_L}{N} \Bigg(\sum_{k=0}^{N-1}\int_{kc_L/N}^{(k+1)c_L/N}\Bigl|\frac{N}{c_L}G_k(t)\Bigr|^2\dz\Bigg)^{1/2}\notag\\
	&\leq C\|K'\|_{L^\infty}\|v\|_{L^\infty}\frac{c_L}{N}\Bigl(1+\sum_{k=0}^{N-1} \int_{kc_L/N}^{(k+1)c_L/N}  \frac{N^2}{c_L^2} G^2_k(t)\dz\Bigr).\label{eq:a3est2}
\end{align}

Then, let us consider the third term in \eqref{eq:a3est}. We have that, by H\"older inequality and a reasoning similar to how we estimated the other two terms
\begin{align}
	\Bigg|&\sum_{k=0}^{N-1} \int_{kc_L/N}^{(k+1)c_L/N} \frac{N}{c_L}G_k(t) \left(\frac{N}{c_L}z - k \right) \frac{c_L}{N} \Bigl( v_{k+1}(t) \sum_{j\neq k+1} K'(x_{k+1} - x_j) - v_k(t) \sum_{j\neq k} K'(x_k - x_j)  \Bigr) \dz\Bigg|\leq\notag\\
	&\leq \Bigg(\sum_{k=0}^{N-1} \int_{kc_L/N}^{(k+1)c_L/N} \Bigl|\frac{N}{c_L}G_k(t) \Bigr|^2\dz\Bigg)^{1/2}\cdot\notag\\
	&\cdot\Bigg(\sum_{k=0}^{N-1} \int_{kc_L/N}^{(k+1)c_L/N}\Bigl(\frac{N}{c_L}z-k\Bigr)^2\frac{c_L^2}{N^2}\Bigl(v_{k+1}(t)\sum_{j\neq k+1} K'(x_{k+1} - x_j) - v_k(t) \sum_{j\neq k} K'(x_k - x_j)  \Bigr)^2\Bigg)^{1/2}\notag\\
	&\leq C\frac{c_L}{\sqrt{N}}\|v\|_{L^\infty}\|K'\|_{L^\infty} \Bigg(\sum_{k=0}^{N-1} \int_{kc_L/N}^{(k+1)c_L/N} \Bigl|\frac{N}{c_L}G_k(t) \Bigr|^2\dz\Bigg)^{1/2}\notag\\
	&\leq C\frac{c_L}{\sqrt{N}}\|v\|_{L^\infty}\|K'\|_{L^\infty}\Bigl(1+\sum_{k=0}^{N-1} \int_{kc_L/N}^{(k+1)c_L/N}  \frac{N^2}{c_L^2} G^2_k(t)\dz\Bigr).\label{eq:a3est3}
\end{align}

From \eqref{eq:a3est1},\eqref{eq:a3est2} and \eqref{eq:a3est3} we conclude that
\begin{align}\label{eq:A4}
	|A_4|&\leq C(\|K'\|_{L^\infty}+\|K''\|_{L^\infty})\|v\|_{L^\infty}\frac{L}{\sqrt{N}}\Bigl(1+\sum_{k=0}^{N-1} \int_{kc_L/N}^{(k+1)c_L/N}  \frac{N^2}{c_L^2} G^2_k(t)\dz\Bigr).
\end{align}

\textbf{Step 4 (Estimate of the term $A_3$)}
Let us now consider the term $A_3$. One has that
\begin{align}
	|A_3|&=\Bigl|\sum_{k=0}^{N-1} \int_{kc_L/N}^{(k+1)c_L/N} K'\ast \rho^N_L (t,X(t,z)) \frac{N}{c_L}\Bigl(\frac{N}{c_L}z - k\Bigr)(G_{k+1}(t)- G_k(t)) \dz\Bigr|\notag\\
	&=\Bigl|\sum_{k=0}^{N-1} G_k(t)\Bigl(\int_{kc_L/N}^{(k+1)c_L/N} K'\ast \rho^N_L (t,X(t,z)) \frac{N}{c_L}\Bigl(\frac{N}{c_L}z - k\Bigr)\dz\notag\\
	&-\int_{(k-1)c_L/N}^{(k)c_L/N} K'\ast \rho^N_L (t,X(t,z)) \frac{N}{c_L}\Bigl(\frac{N}{c_L}z - (k-1)\Bigr)\dz\Bigr) \Bigr|\notag\\
	&=\Bigl|\sum_{k=0}^{N-1} \int_{kc_L/N}^{(k+1)c_L/N}\frac{N}{c_L} G_k(t)\Bigl(K'\ast \rho^N_L (t,X(t,z))-K'\ast \rho^N_L (t,X(t,z-c_L/N))\Bigr)\Bigl(\frac{N}{c_L}z - k\Bigr)\dz\Bigr|\notag\\
	&\leq C\|K''\|_{L^\infty}\|v\|_{L^\infty}\frac{L}{\sqrt{N}}\Bigg(\sum_{k=0}^{N-1} \int_{kc_L/N}^{(k+1)c_L/N} \Bigl|\frac{N}{c_L}G_k(t) \Bigr|^2\dz\Bigg)^{1/2}\notag\\
	&\leq C\|K''\|_{L^\infty}\|v\|_{L^\infty}\frac{L}{\sqrt{N}}\Bigl(1+\sum_{k=0}^{N-1} \int_{kc_L/N}^{(k+1)c_L/N}  \frac{N^2}{c_L^2} G^2_k(t)\dz\Bigr).\label{eq:A3}
\end{align}

\textbf{Step 5 (Estimate of the term $A_2$)}
	
	Let us now consider the term $A_2$. By the H\"older inequality and using the fact that $K'\in L^\infty(\R)$, one has that
	
	\begin{align}
		|A_2|&\leq \Biggl(\sum_{k=0}^{N-1}\int_{kc_L/N}^{(k+1)c_L/N}\frac{N^2}{c_L^2}G_k^2(t)\dz\Biggr)^{1/2}\cdot\Bigl(\int_0^{c_L}\bigl(K'\ast\rho^N_L(t,X(t,z))\bigr)^2\dz\Bigr)^{1/2}(1+\|v\|_{L^\infty})\notag\\
		&\leq c_L\|K'\|_{L^\infty}(1+\|v\|_{L^\infty})\Biggl(\sum_{k=0}^{N-1}\int_{kc_L/N}^{(k+1)c_L/N}\frac{N^2}{c_L^2}G_k^2(t)\dz\Biggr)^{1/2}.\label{eq:A2}
	\end{align}

\textbf{Step 6 (Conclusion)}

From \eqref{eq:ender}, \eqref{eq:A5}, \eqref{eq:A4},  \eqref{eq:A3} and \eqref{eq:A2} we obtain, for $N$ sufficiently large, 
\begin{align}
	\frac{d}{dt}\Fcal^v_L(\rho^N_L(t))&\leq -\sum_{k=0}^{N-1}\int_{kc_L/N}^{(k+1)c_L/N}\frac{N^2}{c_L^2}G_k^2(t)\dz \notag\\
	&+C_1\bigl(\|K'\|_{L^\infty}, \|K''\|_{L^\infty}, \|v\|_{L^\infty})\frac{L}{\sqrt{N}}\Bigl(1+\sum_{k=0}^{N-1}\int_{kc_L/N}^{(k+1)c_L/N}\frac{N^2}{c_L^2}G_k^2(t)\dz\Bigr)\notag\\
	& +C_2\bigl(\|K'\|_{L^\infty},\|v\|_{L^\infty}\bigr)\Biggl(1+\Biggl(\sum_{k=0}^{N-1}\int_{kc_L/N}^{(k+1)c_L/N}\frac{N^2}{c_L^2}G_k^2(t)\dz\Biggr)^{1/2}\Biggr).
\end{align} 
\end{proof}
% conclude that
%\begin{align}
%	\frac{d}{dt}\Fcal_L(\rho^N_L(t))&\leq  - \sum_{k=0}^{N-1} \int_{kc_L/N}^{(k+1)c_L/N} \Bigl[ \frac{N}{c_L} (\phi(\rho_k(t))-\phi(\rho_{k-1}(t)) + K'\ast \rho^N_L (t,X(t,z)) \Bigr]^2v_k(t) \dz \notag\\
%	&+ C_1\frac{L}{\sqrt{N}}+C_2,
%\end{align} 
%where $C_1=C_1\bigl(\|K'\|_{L^\infty}, \|K''\|_{L^\infty}, \|v\|_{L^\infty})$ and $C_2=C_2\bigl(\|K'\|_{L^\infty},\|v\|_{L^\infty}\bigr)$.
%\end{proof}
\section{$L^\infty$ bound}

The aim of this section is to prove Theorem \ref{thm:linfty}.

The proof is quite different from the one given in \cite{DRR} which deals with  the case of constant nonlinear mobility. It is based on an $L^\infty$ bound on the function $\phi_v(\rho^N_L)$.

The following lemma guarantees the existence of a minimal time $T_0>0$ on which the deterministic particle approximation is well-defined (depending only on the $L^\infty$ norm of the initial datum).
\begin{lemma}\label{lemma:mininterv}
	Let $\rho_{0,L}\in L^\infty(\T_L)$, with $\int_{\T_L}\rho_{0,L}=c_L$. Then there exists $T_0=T_0(\|\rho_{0,L}\|_{L^\infty},\|K'\|_{L^\infty},\|v\|_{L^\infty})>0$ such that the deterministic particle approximation $\{\rho^N_L\}_{N\in\N}$ defined starting from $\rho_{0,L}$ is well-defined on $[0,T_0)$.
\end{lemma}

The proof is analogous to Lemma 4.6 in \cite{DRR}, therefore we omit it. 

In the proof of Theorem \ref{thm:linfty} below we show that $T_0$ can be prolonged to be larger than any fixed $T>0$, provided the number of particles in the approximation is sufficiently large.

\begin{proof}
	[Proof of Theorem \ref{thm:linfty}: ] By definition of the deterministic particle approximation $\rho^N_L$ and the ODE \eqref{eq:ode} one has that
	\begin{align}
		\partial_t\phi_v(\rho_k)&=-\frac{N}{c_L}\phi'(\rho_k)v(\rho_k)\rho_k^2(\dot{x}_{k+1}-\dot{x}_k)\\&=-\phi'(\rho_k)v(\rho_k)\rho_k^2\Bigl[-v_{k+1}(t)\sum_{j \neq k+1}K'(x_{k+1}(t)-x_j(t))+v_k(t)\sum_{i \neq k}K'(x_k(t)-x_j(t))\notag\\
		&-\frac{N^2}{c_L^2}(G_{k+1}(t)-G_k(t))\Bigr].
	\end{align}
Assume $\rho_k(t)=\max\{\rho_j(t)\}_j$. Then, $\rho_k(t)\geq\max(\rho_{k-1}(t),\rho_{k+1}(t))$ and  $v_{k+1}(t)=v_k(t)=v(\rho_k(t))$ which implies
\begin{align}
		\partial_t\phi_v(\rho_k)&= \phi'(\rho_k)v(\rho_k)\rho_k^2\Biggl[v(\rho_k)\Bigl(\sum_{j \neq k+1}K'(x_{k+1}(t)-x_j(t))-\sum_{i \neq k}K'(x_k(t)-x_j(t))\Bigr)+\frac{N^2}{c_L^2}(G_{k+1}(t)-G_k(t))\Biggr].
\end{align}
Now notice that, by monotonicity of $\phi_v$ and maximality of $\rho_k(t)$, $G_{k+1}(t)-G_k(t)\leq 0$, thus
\begin{align}
	\partial_t\phi_v(\rho_k)&\leq \phi'(\rho_k)\rho_k^2v^2(\rho_k)\Bigl[\sum_{j \neq k+1,k}[K'(x_{k+1}(t)-x_j(t))-K'(x_k(t)-x_j(t))]+2K'(x_{k+1}(t)-x_k(t))\Bigr].
\end{align}
Hence, since $K'$ is Lipschitz outside the origin and moreover   $\phi'(\rho)\rho\leq C\phi(\rho)$ and $\rho\leq \phi(\rho)$ whenever $\rho\geq 1$ (as we can assume it is the case for $\rho_k$), 
one has that
\begin{align}
		\partial_t\phi_v(\rho_k)&\leq CN(x_{k+1}(t)-x_k(t))\rho_k(t)\|K''\|_{L^\infty}\phi(\rho_k)v^2(\rho_k)+2\|K'\|_{L^\infty}\phi^2(\rho_k)v^2(\rho_k)\notag\\
		&\leq C\|K''\|_{L^\infty}\|v\|_{L^\infty}\phi(\rho_k)v(\rho_k)+2\|K'\|_{L^\infty}\phi^2(\rho_k)v^2(\rho_k).\label{eq:ineqquad}
\end{align}
Now notice that
\begin{equation}
	\phi_v(\rho)=\int_0^{\rho}\phi'(s)v(s)\ds=\phi(\rho)v(\rho)-\int_0^{\rho}\phi(s)v'(s)\ds\geq\phi(\rho)v(\rho),
\end{equation}
thus \eqref{eq:ineqquad} becomes
\begin{align}
	\partial_t\phi_v(\rho_k)&\leq C\|K''\|_{L^\infty}\|v\|_{L^\infty}\phi_v(\rho_k)+2\|K'\|_{L^\infty}\phi_v^2(\rho_k).
\end{align}
In order to control the second tern in the r.h.s. of \eqref{eq:ineqquad}, we use the triangle inequality,  and get
\begin{align}
		\partial_t\phi_v(\rho_k)
		&\leq C\|K''\|_{L^\infty}\|v\|_{L^\infty}\phi_v(\rho_k)\notag\\
		&+4\|K'\|_{L^\infty}\Bigl[\phi^2_v(\rho_{k_0})+N\sum_{j=k_0+1}^k(\phi_v(\rho_{j})-\phi_v(\rho_{j-1}))^2\Bigr],
\end{align}
for some $k_0\leq k$. In particular, choosing $k_0=k_0(N,t)$ such that $\rho_{k_0}\leq \bar C$ for a fixed constant $\bar C\geq c_L/L$ (thanks to the relation $\sum_{j=0}^{N-1}(x_{j+1}-x_j)=L$), and integrating the above inequality in time between $t_0$ and $t$, in order to bound $\phi_v(\rho^N_L)$ in $L^\infty$ one is left to bound the quantity
\begin{equation}\label{eq:q1}
N^2\int_{t_0}^t\sum_{j=k_0+1}^k\int_{c_Lj/N}^{c_L(j+1)/N}(\phi_v(\rho_j(s))-\phi_v(\rho_{j-1}(s)))^2\dz\ds.
\end{equation}
A uniform bound for \eqref{eq:q1} is given by \eqref{eq:stimaw12} in Corollary \ref{cor:energybound}, thus giving a uniform $L^\infty$ bound on $[t_0,t]\times[0,L]$ for the functions $\{\phi_v(\rho^N_L(t,x))\}$.

In order to get an $L^\infty$ bound for the functions $\rho^N_L$, let us now distinguish two cases: either
\begin{equation}\label{eq:1phi}
	\phi_v(s)\to+\infty\quad \text{for $s\to+\infty$},
\end{equation}
or 
\begin{equation}\label{eq:2phi}
	\exists \,\bar c<+\infty:\quad\phi_v(s)\leq\bar c\quad\forall\,s.
\end{equation}
If \eqref{eq:1phi} holds, then the $L^\infty$ bound on $\phi_v(\rho^N_L)$ implies an $L^\infty$ bound on $\rho^N_L$.

If instead \eqref{eq:2phi} holds, our first claim is that 
\begin{equation}
	\label{eq:boundsvs}
	\exists\,c:\quad sv(s)\leq c,\quad\forall\,s.
\end{equation}
Indeed, letting $f(s):=sv(s)$, one has that
\begin{align}
	f(s)&\leq f(0)+\int_0^sf'(\tau)\d\tau\notag\\
	&\leq\int_0^s\bigl[v(\tau)+\tau v'(\tau)\bigr]\d\tau\notag\\
	&{\leq}\int_0^sv(\tau)\d\tau\notag\\
	&\leq\phi_v(s)+\|v\|_{\infty}\notag\\
	&\leq \bar c+\|v\|_{\infty},
\end{align}
where we used that $v'\leq0$ and \eqref{eq:2phi}.
Thus reasoning as above 
for a maximum value $\rho_k$
\begin{align}
	\partial_t\rho_k&=-N\rho_k^2(\dot{x}_{k+1}-\dot{x}_k)\leq\rho_k^2v(\rho_k)\Bigl(\sum_{j \neq k+1}K'(x_{k+1}(t)-x_j(t))-\sum_{i \neq k}K'(x_k(t)-x_j(t))\Bigr)\notag\\
	&\leq CN(x_{k+1}(t)-x_k(t))\rho_k(t)v(\rho_k(t))\rho_k(t)+2\|K'\|_{\infty}\rho_k^2(t)v(\rho_k(t))\notag\\
	&\leq c_1+c_2\rho_k(t),
\end{align}
thus $\rho_k(t)\leq \bar C(K,\|v\|_{\infty},T)$ for all $t\leq T$ and the proof of the Theorem is completed.

\end{proof}

\section{Compactness in the $L^1$ topology}.

In this section we discuss the strong $L^1$-compactness in space and time of the following functions:
\begin{itemize}
	\item $\{\phi_v(\rho^N_L)\}_{N\geq\bar N}$  (see Theorem \ref{thm:compactness});
	\item $\{\phi_v(\rho_L)\}_{L>0}$  (see Theorem~\ref{thm:compactness2}).
\end{itemize} 
Under the $L^\infty$ bound on $\rho^N_L$ proved in Theorem \ref{thm:linfty} and the monotonicity of $\phi_v$, the above will also imply strong $L^1$-compactness of the functions $\rho^N_L$ and  $\rho_L$.

In order to show compactness of the approximate solutions we will use the following generalized Aubin-Lions Lemma given in Theorem 2 of~\cite{RossiSav}. Before recalling it, let us introduce the following definitions. 

Let $X$ be a separable Banach space.  We recall that a functional $\mathcal G: X\to[0,+\infty]$ is a \emph{normal integrand} if it is l.s.c. with respect to the Borel $\sigma$-algebra $\mathcal B(X)$.
$\mathcal G$ is also \emph{coercive} if the sublevels $\{v \in X : \mathcal G (v) \leq c\}$ are compact for any $c\geq0$. 

A pseudo-distance $g : X \times X \to [0, +\infty]$ is compatible with $\mathcal G$  if for every $v,w$ such that $g(v,w)=0$ and $\mathcal G(v)<+\infty$, $\mathcal G(w)<+\infty$ then $v=w$.

We now recall Theorem 2 of~\cite{RossiSav} in a simplified form which is sufficient for our purposes.

\begin{theorem}\cite{RossiSav}\label{thm:aubinlions}
	Let $X$ be a separable Banach space. Let $\mathcal U$ be a set  of measurable functions $u:(0,T)\to X$, let $\mathcal G : X \to [0, +\infty]$ be a normal coercive integrand  and 	$g : X \times X \to [0, +\infty]$ be a l.s.c. pseudo-distance compatible with $\mathcal G$. Assume moreover that
	\begin{equation}\label{eq:tight}
		\sup_{u\in\mathcal U}\int_0^T\mathcal G(u(t))\dt<+\infty
	\end{equation}
	and
	\begin{equation}\label{eq:gholder}
		\lim_{h\to0}\sup_{u\in\mathcal U}\int_0^{T-h}g(u(t+h),u(t))\dt=0.
	\end{equation}
	Then $\mathcal U$ contains a sequence $u_n$ which converges in measure (w.r.t. $t\in(0,T)$ and with values in $X$) to a limit $u:~(0,T)\to~X$.
	
\end{theorem}

Let us now fix $C_0,C_1>0$ and let us consider any $\rho_0:\R\to[0,+\infty)$ such that $\|\rho_0\|_{L^\infty}\leq C_1/4$, $\Fcal^v(\rho_0)\leq C_0/4$ and $\int_{\R}\rho_0=1$, where
\begin{equation}
\Fcal^v(\rho)=\frac{1}{2}\int_{\R}K\ast\rho(x)\rho(x)\dx+\int_{\R}W_v(\rho(x))\dx.
\end{equation}

For any measurable function $f:\R\to[0,+\infty)$ we define 
\begin{equation}\label{eq:gL}
	(f)_L:[-L/2,L/2]\to[0,+\infty),\qquad (f)_L(x)=f(x)\chi_{[-L/2,L/2]}(x).
\end{equation}
With a slight abuse of notation in this subsection we will still denote with $(f)_L$ the corresponding periodic extension function on $\T_L$. 

Let $\rho^N_L$ be the deterministic particle approximation on $\T_L$ starting from the density $(\rho_{0})_L$, with $(\rho_0)_L$ defined as~\eqref{eq:gL} from $\rho_0$. W.l.o.g., we can assume that $N$ and $L$ are sufficiently large so that
\begin{align}\label{eq:lambdacond}
	\sup_{N,L}\bigl\|\rho^N_L(0)\bigr\|_{L^\infty}\leq C_1,\qquad \sup_{N,L}\Fcal^v_L(\rho^N_L(0))\leq C_0.
\end{align}
Indeed, the functionals $\Fcal^v_L$ $\Gamma$-converge to $\Fcal^v$ as $L\to+\infty$. 

Theorem~\ref{thm:linfty} guarantees that for every $T>0$ the functions $\rho^N_L$ are well defined on $[0,T]$ as soon as $N\geq\bar N$ with $\bar N=\bar N(T, C_0,L,K,\|v\|_{L^\infty})$ and that they enjoy the following bound

\begin{equation}\label{eq:gamma1gamma2}
	\sup_{N\geq\bar N}\sup_{t\in[0,T]}\|\rho^N_L(t)\|_{L^\infty(\T_L)}\leq\bar C,
\end{equation}
where $\bar C=\bar C(K,C_0,C_1,\|v\|_{L^\infty},T)$. In particular, one can observe the following 
\begin{remark}\label{rem:linfty} The functions $\phi_v(\rho^N_L)$ satisfy the upper bound 
	\begin{equation}\label{eq:linftyboundphi}
		\sup_{N\geq\bar N}\sup_{t\in[0,T]}\|\phi_v(\rho^N_L(t))\|_{L^\infty}\leq \phi_v(\bar C).
	\end{equation}
	%and therefore, also in view of Remark~\ref{rem:lambda}, they  lie in the domain of the functional $\mathcal G$.
\end{remark}

Moreover, let us   define, for a function $u\chi_{[-L/2,L/2]} \in L^1(\R)$, the quantity
\begin{align*}
	TV(u) &: = \sup_{I\in\N\,\,}\sup_{-L/2=x_0<x_1<\dots<x_I=L/2}\sum_{i=0}^{I-1} |u(x_{i+1})-u(x_i)|.
\end{align*}
Let us observe, that $TV(u)$ corresponds to the standard total variation of the function $u$.

First we show the following compactness result.
\begin{theorem}\label{thm:compactness}
	Let $\rho_0\in L^1(\R)\cap L^\infty(\R)$ such that $\Fcal^v(\rho_0)<+\infty$. Then, for all $T>0$ the %deterministic particle approximation $\{\rho_L^N\}_{N\in\N}$ defined on $[0,T]\times\T_L$ starting from $(\rho_0)_L$ converges weakly in $L^\infty([0,T]\times\T_L)$ up to subsequences as $N\to+\infty$ to a function $\rho_L$. Moreover, the 
	functions $\phi_v(\rho^N_L)$ converge strongly in $L^1([0,T]\times\T_L)$ (up to subsequences) to $\phi_v(\rho_L)$, being $\rho_L$ the weak* limit in $L^\infty$ of the deterministic particle approximations $\rho^N_L$. Moreover, the functions $\rho^N_L$ converge strongly in $L^1([0,T]\times\T_L)$ as well. 
\end{theorem}

We will need the following preliminary lemma.  

\begin{lemma}\label{lemma:g}
	Let $\rho_L^N : [0,T] \times \T_L\to(0,\infty)$, $N\geq\bar N(T,L,C_0,K,\|v\|_{L^\infty})$ be the deterministic particle approximations starting from $(\rho_{0})_L$ and $t,s\in[0,T]$. Then there exists a constant $C(C_0,K,\|v\|_{L^\infty})$ independent of $N,L$ such that 
	\begin{equation}\label{eq:timecont}
		\mathtt d_{W_1} \left(\frac{\rho_L^N(s)}{\|\rho_L^N(s)\|_{L^1}}, \frac{\rho_L^N(t)}{\|\rho_L^N(t)\|_{L^1}} \right)\leq C(C_0,K,\|v\|_{L^\infty}) |t-s|^{1/2}.
	\end{equation} 
\end{lemma}

\begin{proof} 
	Denote by $X(\tau)=X(\tau,\cdot)$ the pseudo-inverse of $\rho^N_L(\tau)$ and let $s<t$.  Let $c_L=\int_{\T_L}(\rho_{0})_L=\int_{\T_L}\rho^N_L(\tau)$.
	In order to estimate the $1$-Wasserstein distance of the deterministic particle approximations at different times we use the well-known identity
	\[
	\mathtt d_{W_1} \left({\rho_L^N(s)}, {\rho_L^N(t)} \right)=\|X(s)-X(t)\|_{L^1([0,c_L])}
	\]
	One has that
	\begin{align*}
		\|X(s)&-X(t)\|_{L^1([0,c_L])}^2	\leq c_L \|X(s)-X(t)\|^2_{L^2([0,c_L])}\notag\\
		&\leq c_L\sum_{k=0}^{N-1}\int_{k c_L/N}^{(k+1) c_L/N}\Bigl|x_k(t)-x_k(s)+\Bigl(\frac{N}{c_L}z-k\Bigr)[x_{k+1}(t)-x_{k}(t)-x_{k+1}(s)+x_k(s)]\Bigr|^2\dz\notag\\
		&\leq\sum_{k=0}^{N-1}\frac{9c_L}{N}|x_k(t)-x_k(s)|^2.
	\end{align*}
	Moreover,
	\begin{align*}
		|x_{k}(t)-x_k(s)|^2=\Big|\int_s^t\dot x_k(\tau)\d\tau\Big|^2\leq|t-s|\int_s^t|\dot x_k(\tau)|^2\d\tau,
	\end{align*}
	hence 
	\begin{align}\label{eq:ftauest}
		\|X(s)-X(t)\|_{L^1}^2&\leq 9|t-s|\int_s^t\frac{c_L}{N}\sum_{k=0}^{N-1}|\dot x_k(\tau)|^2\d\tau
	\end{align}
	provided $N\gg L$.

	By the ODE \eqref{eq:ode} and Corollary \ref{cor:energybound},  for $N$ sufficiently large one has that
	\begin{align}
	\int_s^t\frac{c_L}{N}\sum_{k=0}^{N-1}|\dot x_k(\tau)|^2\d\tau&\leq C(\|K'\|_{L^\infty}, \|v\|_{L^\infty})|t-s|+\int_s^t\sum_{k=0}^{N-1}\int_{kc_L/N}^{(k+1)c_L/N}N^2\bigl(\phi_v(\rho_k(\tau))-\phi_v(\rho_{k-1}(\tau))\bigr)^2\dz\d\tau\notag\\
	&\leq \tilde C_1(\|K'\|_{L^\infty}, \|v\|_{L^\infty}, )(1+|t-s|)+\tilde C_2\Fcal^v_L(\rho_{0,L}).
	\end{align}
Hence, since by assumption $\Fcal^v_L(\rho_{0,L})\leq C_0$, \eqref{eq:timecont} holds.

\end{proof}

In particular one can observe the following

\begin{remark}
	\label{rem:lambda} Let us assume that $\rho_0$ has finite first moments. 
	Since the functions $\rho^N_L(0)$ converge in $L^1(\T_L)$ to $(\rho_{0})_L$, which in turn as $L\to+\infty$ converge to $\rho_0$ in $L^1(\R)$,  by Lemma~\ref{lemma:g} 
	\[
	\mathtt d_{W_1}\Big(\frac{(\rho^N_L)_L(t)}{\|(\rho^N_L(t))_L\|_{L^1}}, \rho_0\Big)\leq C(C_0,K,\|v\|_{L^\infty})T^{1/2}+\tilde C,\quad\forall\,t\in[0,T]
	\] 
	with $\tilde C$ independent of $N,L$ as soon as $N\geq \bar N$ is sufficiently large. %Therefore the constant $\Lambda$ in definition~\eqref{eq:Gdef} of $\mathcal G$  can be chosen such that 
	%	\begin{equation*}
		%	\Lambda\geq 2(C(C_0,K)T^{1/2}+\bar C),
		%	\end{equation*}
	%	so that all the deterministic particle approximations (and their limits as $N\to\infty$ on $[0,T]$, by lower semicontinuity of $\mathtt d_{W_1}$) satisfy the condition 
	%	\[
	%		\mathtt d_{W_1}\Big(\frac{\rho^N_L(t)}{\|\rho^N_L(t)\|_{L^1}}, \rho_0\Big)\leq \Lambda.
	%	\]
	%	The above condition  is necessary for $\mathcal G$ to be finite along the deterministic particle evolution. Moreover, 
	
	The above condition gives tightness of the deterministic particle approximations. Thanks to the fact that $\phi_v(\rho)\leq c\rho$ if $\rho\leq 1$, this converts into tightness for the functions $(\phi_v(\rho^N_L))_N$.
	
\end{remark}
We are now ready to prove Theorem \ref{thm:compactness}. 

\begin{proof}[Proof of Theorem \ref{thm:compactness}: ]
	
	Let us define a functional $\mathcal G$ as follows
	\begin{equation}\label{eq:Gdef}
		\mathcal G(u)= TV(u)+\|u\|_{L^1(\R)}+\mathbb 1_{\{\|u\|_{L^\infty}\leq \phi_v(\bar C)\}}+\mathbb 1_{\Big\{\mathtt d_{W_1}\Big(\frac{(\phi_v)^{-1}(u)}{\|(\phi_v)^{-1}(u)\|_{L^1}}, \rho_0\Big)\leq \Lambda\Big\}},
	\end{equation}
	where 
	\begin{equation*}
		\mathbb 1_A(x)=\left\{\begin{aligned}
			&0 && &\text{if $x\in A$}\\
			&+\infty && &\text{if $x\in A^c$},
		\end{aligned}\right.
	\end{equation*}
	$\mathtt d_{W_1}$ is the standard 1-Wasserstein distance between probability measures and $\Lambda$ is a positive constant satisfying 
	\begin{equation*}
		\Lambda\geq 2(C(C_0,K,\|v\|_{L^\infty})T^{1/2}+\tilde C),
	\end{equation*}
	where $C$, $\tilde C$ are the constant defined in Remark \ref{rem:lambda}. In particular,  all the deterministic particle approximations (by lower semicontinuity of $\mathtt d_{W_1}$) satisfy the condition 
	\[
	\mathtt d_{W_1}\Big(\frac{(\rho^N_L(t))_L}{\|(\rho^N_L(t))_L\|_{L^1}}, \rho_0\Big)\leq \Lambda.
	\]

	Moreover, define
	\[
	g(u,w)=\mathtt d_{W_1} \left(\frac{(\phi_v)^{-1}(u)}{\|(\phi_v)^{-1}(u)\|_{L^1}},\frac{(\phi_v)^{-1}(w)}{\|(\phi_v)^{-1}(w)\|_{L^1}}\right).
	\] 
	
	It is fairly easy to see that $\mathcal G$ is a normal coercive integrand. Indeed, the compactness of the sublevels in $L^1_{\loc}$ comes from the first three terms of $\mathcal G$ and gets upgraded to compactness in $L^1$ due to the tightness given by the last term in the definition of $\mathcal G$ and to the condition~$\phi_v(\rho)\leq c\rho$ for $\rho\leq1$. Moreover, $g$ is a l.s.c. pseudo-distance compatible with $\mathcal G$.

	We now apply Theorem~\ref{thm:aubinlions} to the set of functions 
	\[	\mathcal U=\{\phi_v(\rho^N_L)\}_{N\in\N}
	\] and the functionals $\mathcal G$ and $g$ defined above.

	Concerning property~\eqref{eq:gholder}, it follows directly from Lemma \ref{lemma:g}.

	Let us now prove~\eqref{eq:tight} for the functional $\mathcal G$ defined in~\eqref{eq:Gdef} on the functions $\phi_v(\rho^N_L)$. First of all, using the following
	\begin{align}
		\phi_v(s)&=\phi(s)v(s)-\int_0^s\phi(z)v'(z)\dz\notag\\
		\|v'\|_{\infty}&<+\infty	\end{align}
	one has that 
	\begin{align}
		\sup_{N\geq\bar N}\int_0^T\int_{\T_L}\phi_v(\rho^N_L(t,x))\dx\dt&\leq C(\|v\|_{L^\infty}, \|\rho^N_L\|_{L^\infty},m)	\sup_{N\geq\bar N}\int_0^T\int_{\T_L}\rho^N_L(t,x)\dx\dt\notag\\
		&\leq C(\|v\|_{L^\infty}, \|\rho^N_L\|_{L^\infty},m)T,\label{eq:phil1bound}
	\end{align}
where $m$ is the power of the nonlinear diffusion, 
	and one can use Theorem \ref{thm:linfty} to bound $\|\rho^N_L\|_{L^\infty}$.
	Moreover, one has the following
	
	\begin{equation}\label{eq:bvest}
		\sup_{N\geq\bar N}\int_0^T TV(\phi_v(\rho^N_L(t)))\dt \leq C(K, C_0,\|v\|_{L^\infty})(1+T).
	\end{equation}
	
	Indeed, this follows by Corollary \ref{cor:energybound} and by the  inequality
		\begin{equation}\label{eq:tvtv2ineq}
		\sum_{k=0}^{N-1}|a_{k+1}-a_k|\leq\max\Biggl\{1,N\sum_{k=0}^{N-1}|a_{k+1}-a_k|^2\Biggr\}.
	\end{equation}

	From Lemma~\ref{lemma:g}, Remarks~\ref{rem:lambda} and~\ref{rem:linfty} and the bounds~\eqref{eq:phil1bound} and~\eqref{eq:bvest} we deduce that for every $L$ the set
	\[
	\mathcal U=\{\phi_v(\rho^N_L)\}_{N\geq\bar N(T,L, C_0,K,\|v\|_{L^\infty})}
	\]
	satisfies the assumptions of Theorem~\eqref{thm:aubinlions} on $X=L^1(\T_L)$. Hence Theorem~\ref{thm:aubinlions} can be applied, implying the convergence in measure (w.r.t. $t$ with values in $L^1(\T_L)$ and up to subsequences) of the functions $\phi_v(\rho^N_L):(0,T)\times \T_L\to\R$ to a function $\bar{\phi}_{v,L}$. 
	
	By the $L^\infty$ bound~\eqref{eq:linftyboundphi}, the above convergence can be upgraded to strong convergence in $L^1([0,T]\times\T_L)$. 
	
	By Theorem~\ref{thm:linfty}, also the sequence  $\{\rho^N_L\}_{N\geq \bar N}$ is uniformly bounded on $[0,T]$. In particular, the functions  $\rho^N_L$ converge weakly* in $L^\infty([0,T]\times\T_L)$ (up to subsequences) to some bounded function $\rho_L$. Thus by the strict monotonicity and continuity of $\phi_v$ one has that  $\rho^N_L$ converge strongly in $L^1$ to $\rho_L$ and   $\bar{\phi}_{v,L}=\phi_v(\rho_L)$.

	This concludes the proof of Theorem~\ref{thm:compactness}.
\end{proof}
\vskip 0.3 cm

We conclude  this section showing  that also the following compactness result holds.

\begin{theorem}
	\label{thm:compactness2} Let $\rho_0\in L^1(\R)\cap L^\infty(\R)$ with finite first moments and let  $\rho_L:[0,T]\times[-L/2,L/2]\to[0,+\infty)$ be the functions obtained in Theorem~\ref{thm:compactness}. Then, as $L\to+\infty$, $\rho_L$ converge (up to subsequences) strongly in $L^1([0,T]\times\R)$ to   $\rho:[0,T]\times\R\to[0,+\infty)$, where $\rho:[0,T]\times\R\to[0,+\infty)$ is the weak* limit in $L^\infty$ of the functions $\rho_L$. 
	
\end{theorem}

\begin{proof}
	
	In order to prove Theorem~\ref{thm:compactness2} it is sufficient to observe that the estimates of Lemma~\ref{lemma:g}, Remarks~\ref{rem:lambda} and~\ref{rem:linfty} and the upper bounds~\eqref{eq:phil1bound} and~\eqref{eq:bvest} do not depend on $N\geq\bar N,L$ but only on  $C_0$, $C_1$ as in~\eqref{eq:lambdacond}.  Moreover, by lower semicontinuity of the total variation and of the $1$-Wasserstein distance $\mathtt d_{W_1}$ such estimates  uniformly hold also for the sequences $\rho_L$,  $\phi_v(\rho_L)$. Hence it is possible to apply Theorem~\ref{thm:aubinlions} to the sequence $\phi_v(\rho_L)$ on $[0,T] \times \R$ and repeat the previous reasoning obtaining a strong $L^1$ limit $\phi_v(\rho):[0,T]\times\R\to[0,+\infty)$ with $\rho_L$ converging to $\rho$ strongly in $L^1$.

\end{proof}

%\begin{remark}
%	Notice that, while in the first limit (namely $L$ fixed and $N\to\infty$) we could have restricted to functions in $X=L^1(\T_L)$ and avoided the last term in the definition~\eqref{eq:Gdef} of $\mathcal G$, in the limit as $L\to\infty$  such a term becomes essential as the supports of the functions $\rho^\lambda_L$ become larger and larger and at the same time $L^1$-compactness of the sublevels of $\mathcal{G}$ is needed.  
%\end{remark}

\section{Convergence to solutions of the  PDE}\label{sec:limit}

Our first goal is to prove Theorem~\ref{thm:convpdelfixed}. One of the differences w.r.t. \cite{DRR} is that for general  non-constant mobilities a minimum principle for initial densities bounded from below does not hold, namely the approximate densities $\{\rho^N_L(t)\}$ may converge to zero as $N\to +\infty$ on some space intervals even if $\{\rho^N_L(0)\}$ is uniformly bounded from below by a positive constant. In particular, an analogue of Lemma 5.8 in \cite{DRR} does not hold. This leads to an additional error term in the limit PDE which is concentrated on the vacuum regions.

\begin{proof}[Proof of Theorem~\ref{thm:convpdelfixed}:]
	Let $\rho_{0,L}$ an $L^1$ density on $\T_L$, and let $c_L=\int_{T_L}\rho_{0,L}\leq1$. W.l.o.g. we can assume that $\rho_{0,L}=(\rho_0)_L$ for some $\rho_0\in L^1(\R)\cap L^\infty(\R)$. In particular, if $N$ is sufficiently large, one has that
	\begin{equation*}
		\sup_{N}\|\rho^N_L(0)\|_{L^\infty(\T_L)}\leq C_1,\qquad\sup_N\Fcal^v_L(\rho^N_L(0))\leq C_0,
	\end{equation*}
	where $\rho^N_L$ is the deterministic particle approximation on $\T_L$ starting from $\rho_{0,L}$ (see~\eqref{eq:lambdacond}). In particular, Theorem~\ref{thm:compactness} guarantees that, up to subsequences, $\rho^N_L$ strongly converges in $L^1([0,T]\times \T_L)$ to a density $\rho_L\in L^1([0,T]\times\T_L)$.

	Remember that, by~\eqref{eq:pdeapprox}, for every fixed $N \in \N, L>0$ the piecewise constant function $\rho^N_L(t,x)$ satisfies, for every $\varphi \in C^\infty_c((0,T)\times \T_L)$, the following equation	
	\begin{equation}\label{eq:pde3}
		\int_0^T \int_{\T_L} \rho^N_L(t,x) \partial_t \varphi(t,x) - \rho^N_L(t,x) \big[{K^\lin}'(\rho^N_L,v^N_L) \ast \rho^N_L(t,x) + \frac{N}{c_L}\phi_v^{\mathrm{lin}}(\rho^N_L(t,x)) \big]\partial_x \varphi(t,x) \dx\dt = 0
	\end{equation}
	where we defined 
	\begin{align*}
		\phi_v^{\mathrm{lin}}(\rho^N_{L}(t,x)) := \sum_{k=0}^{N-1} \chi_{[x_k(t), x_{k+1}(t))} (x) \Bigl[ G_k(t) + \Bigl(\frac{N}{c_L}M_{\rho^N_L}(x) - k\Bigr)(G_{k+1}(t) - G_k(t)) \Bigr]
	\end{align*}
	and $G_k(t) = \phi_v(\rho_k(t)) - \phi_v(\rho_{k-1}(t))$. 
	
	In the following we want to show that as $N\to+\infty$ the equation~\eqref{eq:pde3} converges to the PDE \eqref{eq:pdel} for the density $\rho_L$.

	In particular  we will prove  that 	as $N$ tends to $+\infty$
	\begin{align}
		&\int_0^T \int_{\T_L} \big(\rho^N_L(t,x) - \rho_L(t,x)\big) \partial_t \varphi(t,x) \dx \dt \longrightarrow 0,\label{eq:lim1}\\
		&{\int_0^T \int_{\T_L} \big((\rho^N_L(t,x){K^\lin}'(\rho^N_L,v^N_L) \ast \rho^N_L(t,x) - \rho_{L}K' \ast \rho_L (t,x)v(\rho_L)(t,x) \big) \partial_x \varphi(t,x) \dx \dt \longrightarrow 0,}\label{eq:lim2}
	\end{align}
	and that
	\begin{align}
		\lim_{N\to\infty}\int_0^T \int_{\T_L}\rho^N_L(t,x)\frac{N}{c_L} \phi_v^{\mathrm{lin}}(\rho^N_L(t,x)) \partial_x \varphi(t,x) \dx \dt& =- \int_0^T \int_{\T_L} \Phi_{v,L}(t,x)\partial_{xx} \varphi(t,x) \dx \dt,\label{eq:lim3}
	\end{align}	
where $\Phi_{v,L}( t,x)$ will be implicitly defined later (see \textbf{Step 2}) and is such that $\Phi_{v,L}( t,x)=\phi_v(\rho_L(t,x))$ for a.e. $x\in\{\rho_L(t)>0\}$.

	The convergence in~\eqref{eq:lim1} is an immediate consequence of the $L^\infty$ weak* compactness of the densities $\rho^N_L$ implied by Theorem~\ref{thm:linfty}.

\textbf{Step 1}

	Let us focus on the convergence in~\eqref{eq:lim2}. Simple computations lead to the equivalent expression
	\begin{align}
		\int_0^T \int_{\T_L} &\big(\rho^N_L(t,x){K^\lin}'(\rho^N_L,v^N_L) \ast \rho^N_L(t,x) - \rho_LK'\ast \rho_L (t,x)v(\rho_L(t,x)) \big) \partial_x \varphi(t,x) \dx \dt = \notag\\
		= &  \int_0^T \int_{\T_L} (\rho^N_L(t,x) - \rho_L(t,x)) K' \ast \rho_L(t,x)v(\rho_L(t,x)) \partial_x \varphi(t,x) \dx \dt\notag \\
		& + \int_0^T \int_{\T_L} \rho^N_L(t,x) K'\ast\big( \rho^N_L(t,x) - \rho_L (t,x)\big)v(\rho_L(t,x)) \partial_x \varphi(t,x) \dx \dt \notag\\
		&+ \int_0^T \int_{\T_L} \rho^N_L(t,x) K'\ast \rho^N_L(t,x) (v^N_L(t,x)-v(\rho_L(t,x))) \partial_x \varphi(t,x) \dx \dt \notag\\
		& + \int_0^T \int_{\T_L} \rho^N_L(t,x) \big({K^\lin}'(\rho^N_L,v^N_L)\ast\rho^N_L(t,x) - K'  \ast \rho^N_L(t,x)v^N_L(t,x)\bigr) \partial_x \varphi(t,x) \dx \dt\label{eq:convkernel}
	\end{align}
	The first and the second term of the r.h.s. of the above converge to $0$ as $N \to \infty$ because of the $L^\infty$ weak* compactness of the $\rho^N_L$.
	
	As for the third term of \eqref{eq:convkernel}, first notice that
	\begin{equation}
		\int_0^T \int_{\T_L} \rho^N_L(t,x) K'\ast \rho^N_L(t,x) (v(\rho^N_L(t,x))-v(\rho_L(t,x))) \partial_x \varphi(t,x) \dx \dt \quad\longrightarrow\quad0,
	\end{equation}
since $\rho^N_L\to\rho_L$ in any $L^p$ with $p<\infty$ and $v$ is Lipschitz by assumption. Hence, to estimate the third term of \eqref{eq:convkernel} we need to show that
	\begin{equation}\label{eq:conv30}
	\int_0^T \int_{\T_L} \rho^N_L(t,x) K'\ast \rho^N_L(t,x) (v^N_L(t,x)-v(\rho^N_L(t,x))) \partial_x \varphi(t,x) \dx \dt \quad\longrightarrow\quad0.
\end{equation}
	As observed above, since $\phi_v(\rho_L(t))\in BV(\T_L)$ for a.e. $t$ and $\phi_v$ is a strictly monotone function, $\rho_L(t)$ is continuous up to a countable number of points $J_L(t)$ where it has a jump. Moreover, it is fairly easy to see that $\rho^N_L(t)$ is locally in $BV$ at a point $x$ whenever it is bounded from below by a positive constant. For a.e. point $x$ in $\T_L\setminus J_L(t)$, $\rho^N_L(t,x)\to\rho_L(x)$. Let $x\in\T_L\setminus J_L(t)$ such that $\rho_L(t,x)>0$ and $\rho^N_L(t,x)\to\rho_L(t,x)$. If $x\in[x^N_{k(N,x)},x^N_{k(N,x)+1})$, then $x^N_{k(N,x)+1}-x^N_{k(N,x)}\to0$ and $x^N_{k(N,x)}-x^N_{k(N,x)-1}\to0$ with $\rho^N_{k(N,x)}(t),\rho^N_{k(N,x)-1}(t)\to\rho_L(t,x)$. Thus $v^N_L(t,x)=v_{k(N,x)}(t)=v\bigl(\max(\rho^N_{k(N,x)}(t),\rho^N_{k(N,x)-1}(t))\bigr)\to v(\rho_L(t,x))$. Where $\rho_L(t,x)=0$, then $v^N_L$ might not have the same limit of $v(\rho^N_L)$, but the factor $\rho^N_L$ in the integrand of the third term of \eqref{eq:convkernel} tends to $0$, thus by dominated convergence we can conclude that \eqref{eq:conv30} holds.

	Let us now proceed to estimate the last term in \eqref{eq:convkernel}. Analogously to the proof of Lemma \ref{lemma:kbounds}, we are left to prove that 
	\begin{align}
		\Bigg|\sum_{k=0}^{N-1} &\int_{kc_L/N}^{(k+1)c_L/N} \rho^N_L(t, X(t,z)) \left(\frac{N}{c_L}z - k \right)\cdot\notag\\
		&\cdot \frac{c_L}{N} \Bigl( v_{k+1}(t) \sum_{j\neq k+1} K'(x_{k+1} - x_j) - v_k(t) \sum_{j\neq k} K'(x_k - x_j)  \Bigr)\partial_x\varphi(t,X(t,z)) \dz\Bigg|\longrightarrow0. \label{eq:conv40}
	\end{align}
As above, at the points $x$ where $\rho^N_L(t,x)\to\rho_L(t,x)=0$, one can deduce pointwise convergence to $0$ of the whole integrand and use dominated convergence. At the other points $x\in\T_L\setminus J_L(t)$ where $\rho^N_L(t,x)\to\rho_L(t,x)>0$, one has that  $v^N_L(t,x)$ and $v^N_L(t,y_N)$ have the same limit as $N\to+\infty$, where if $x\in[x^N_{k(N,x)},x^N_{k(N,x)+1})$ then $y_N $ is any point belonging to $[x^N_{k(N,x)+1},x^N_{k(N,x)+2})$. Moreover, at such points (being $\rho^N_L\geq\eps(x)>0$ for $N\geq N'$ and being $x$ a point of continuity of $\rho_L>0$) one has  that $x^N_{k(N,x)+1}-x^N_{k(N,x)}\to0$. Thus also in this case the integrand in \eqref{eq:conv40} converges to $0$ and thus by dominated convergence the limit relation in \eqref{eq:conv40} holds.  

\textbf{Step 2} 
	
	Let us now prove~\eqref{eq:lim3}.

	One has that 
	\begin{align*}
		&\int_0^T \int_{\T_L}\rho^N_L(t,x)\frac{N}{c_L} \phi^\lin_v(\rho^N_L(t,x)) \partial_x \varphi(t,x) \dx \dt= \notag\\
		&=\sum_{k=0}^{N-1} \int_{0}^T\int_{x_k}^{x_{k+1}}\frac{\phi_v(\rho^N_L(t,x_{k}))-\phi_v(\rho^N_L(t,x_{k-1}))}{x_{k+1}-x_k}\partial_x \varphi(t,x)\dx\dt\notag\\
		&+\sum_{k=0}^{N-1} \int_{0}^T\int_{x_k}^{x_{k+1}}\Big[\frac{\phi_v(\rho^N_L(t,x_{k+1}))-\phi_v(\rho^N_L(t,x_k))}{x_{k+1}-x_k}-\frac{\phi_v(\rho^N_L(t,x_{k}))-\phi_v(\rho^N_L(t,x_{k-1}))}{x_{k+1}-x_k}\Big]\cdot\notag\\
		&\cdot\Bigl(\frac{N}{c_L} M_{\rho^N_L(t)}(x)-k\Bigr)\partial_x \varphi(t,x)\dx\dt\notag
		%&=:I^N_1+I^N_2.
	\end{align*}

Integrating by parts  and setting 
\begin{equation}
	a_j:={\phi_v(\rho^N_L(t,x_{j}))-\phi_v(\rho^N_L(t,x_{j-1}))}.
\end{equation}
it holds
	\begin{align*}
	&	\sum_{k=0}^{N-1} \int_{0}^T\int_{x_k}^{x_{k+1}}\frac{\phi_v(\rho^N_L(t,x_{k}))-\phi_v(\rho^N_L(t,x_{k-1}))}{x_{k+1}-x_k}\partial_x \varphi(t,x)\dx\dt\notag\\
		&+\sum_{k=0}^{N-1} \int_{0}^T\int_{x_k}^{x_{k+1}}\Big[\frac{\phi_v(\rho^N_L(t,x_{k+1}))-\phi_v(\rho^N_L(t,x_k))}{x_{k+1}-x_k}-\frac{\phi_v(\rho^N_L(t,x_{k}))-\phi_v(\rho^N_L(t,x_{k-1}))}{x_{k+1}-x_k}\Big]\cdot\notag\\
		&\cdot\Bigl(\frac{N}{c_L} M_{\rho^N_L(t)}(x)-k\Bigr)\partial_x \varphi(t,x)\dx\dt=\notag\\
		&= -\int_{0}^T\int_{\T_L}\Biggl\{\sum_{j=0}^{k(N,x)-1}\int_{x_j}^{x_{j+1}}\frac{1}{x_{j+1}-x_j}\Bigl[a_j(t)+\Bigl(\frac{N}{c_L} M_{\rho^N_L(t)}(y)-j\Bigr)(a_{j+1}(t)-a_j(t))\Bigr]\dy\notag\\
		&+\int_{x_{k(N,x)}}^{x}\frac{1}{x_{k(N,x)+1}-x_{k(N,x)}}\Bigl[a_{k(N,x)}(t)+\Bigl(\frac{N}{c_L} M_{\rho^N_L(t)}(y)-k(N,x)\Bigr)(a_{k(N,x)+1}(t)-a_{k(N,x)}(t))\Bigr]\dy\Biggr\}\cdot\notag\\
		&\cdot\partial_{xx}\varphi(t,x)\dx\dt\notag\\
		&=-\int_{0}^T\int_{\T_L}\Biggl\{\sum_{j=0}^{k(N,x)-1}\Bigl[a_j(t)+\frac12(a_{j+1}(t)-a_j(t))\Bigr]+\frac{x-x_{k(N,x)}}{x_{k(N,x)+1}-x_{k(N,x)}}a_{k(N,x)}(t)\notag\\
		&+\int_{x_{k(N,x)}}^{x}\frac{1}{x_{k(N,x)+1}-x_{k(N,x)}}\Bigl(\frac{N}{c_L} M_{\rho^N_L(t)}(y)-k(N,x)\Bigr)(a_{k(N,x)+1}(t)-a_{k(N,x)}(t))\Bigr]\dy\Biggr\}\cdot\notag\\
		&\cdot\partial_{xx}\varphi(t,x)\dx\dt\notag\\
		&=-\int_{0}^T\int_{\T_L}\Biggl\{-\phi_v(\rho^N_L(t,x_{N-1}))+\phi_v(\rho^N_L(t,x_{k(N,x)-1}))\notag\\
		&+\frac{x-x_{k(N,x)}}{x_{k(N,x)+1}-x_{k(N,x)}}\Bigl[\phi_v(\rho^N_L(t,x_{k(N,x)}))-\phi_v(\rho^N_L(t,x_{k(N,x)-1}))\Bigr]\notag\\
		&+\frac12\Bigl[\phi_v(\rho^N_L(t,x_{N-1}))-\phi_v(\rho^N_L(t,x_0))+\phi_v(\rho^N_L(t,x_{k(N,x)}))-\phi_v(\rho^N_L(t,x_{k(N,x)-1}))\Bigr]\notag\\
		&+\int_{x_{k(N,x)}}^{x}\frac{1}{x_{k(N,x)+1}-x_{k(N,x)}}\Bigl(\frac{N}{c_L} M_{\rho^N_L(t)}(y)-k(N,x)\Bigr)(a_{k(N,x)+1}(t)-a_{k(N,x)}(t))\Bigr]\dy\Biggr\}\cdot\partial_{xx}\varphi(t,x)\dx\dt\notag\\
		&=:-\int_{0}^T\int_{\T_L}\hat\Phi_v(\rho^N_L,t,x)\partial_{xx}\varphi(t,x)\dt\dx.
		\end{align*}
	W.l.o.g., we can assume that up to subsequences
	\begin{equation}
		x_{N-1}\to \bar x,\quad \rho^N_L(t,x_{N-1})\to \rho_L(\bar x),\quad \rho^N_L(t,x_{0})\to \rho_L(x_0).
	\end{equation}
Consider first the points $x$ belonging to the set 
\begin{equation}
	\Omega_+(t):=\{x\in\T_L\setminus J_L(t):\,\rho_L(t,x)>0, \,\rho_L(t,x)=\lim_N\rho^N_L(t,x)\},
\end{equation}
where $J_L(t)$ is the set of discontinuity points of $\rho_L(t,\cdot)$. One has that 
\begin{equation}
	\hat\Phi_v(\rho^N_L,t,x)\to\hat c_L(t)+\phi_v(\rho_L(t,x)),\quad	\forall\,x\in\Omega_+(t),
\end{equation}
where we set 
\[
\hat c_L(t):=-\phi_v(\rho_L(t,\bar x))+\frac12\Bigl[\phi_v(\rho_L(t,\bar x))-\phi_v(\rho_L(t,x_0))\Bigr].
\]

Indeed, 	since $\phi_v(\rho_L(t))\in BV(\T_L)$ for a.e. $t$ and $\phi_v$ is a strictly monotone function, $\rho_L(t)$ is continuous up to a countable  number of points $J_L(t)$ where it has a jump. Moreover, it is fairly easy to see that $\rho^N_L(t)$ is locally in $BV$ at a point $x$ whenever it is bounded from below by a positive constant. By the $L^1$-convergence Theorem \ref{thm:compactness}, for a.e. point $x$ in $\T_L\setminus J_L(t)$ one has that $\rho^N_L(t,x)\to\rho_L(x)$. Let $x\in\T_L\setminus J_L(t)$ such that $\rho_L(t,x)>0$ and $\rho^N_L(t,x)\to\rho_L(t,x)$. If $x\in[x^N_{k(N,x)},x^N_{k(N,x)+1})$, then $x^N_{k(N,x)+1}-x^N_{k(N,x)}\to0$ and $x^N_{k(N,x)}-x^N_{k(N,x)-1}\to0$ with $\rho^N_{k(N,x)}(t),\rho^N_{k(N,x)-1}(t),\rho^N_{k(N,x)+1}(t)\to\rho_L(t,x)$.
Notice that
\begin{equation}
	\Bigl|\{x\in\T_L:\rho_L(t,x)>0\}\setminus\Omega_+(t)\Bigr|=0
\end{equation}
Consider now the subset of $\{\rho_L(t,\cdot)=0\}$ given by
\begin{equation}
	\Omega_0(t):=\{x\in\T_L\setminus J_L(t): \, x\in\cap_N[x^N_{k(N,x)},x^N_{k(N,x)+1}),\rho^N_L(t,x)\to\rho_L(t,x)=0\}.
\end{equation} 

Notice that
%\begin{equation}
%	\Bigl|\{x\in\T_L:\rho_L(t,x)=0\}\setminus\Omega_0(t)\Bigr|=0
%\end{equation}
%and
\begin{equation}
	\phi_v(\rho_L(t,x))=0\quad\text{ whenever }\rho_L(t,x)=0.
\end{equation}
One has that, up to subsequences, for a.e. $x\in\Omega_0(t)$
\begin{equation}
	\hat\Phi_v(\rho^N_L,t,x)\to\hat{\Phi}_{v,L}(t,x),
\end{equation}
where $\hat{\Phi}_{v,L}=\hat c_L$ if $\rho^N_L(t,x^N_{k(N,x)+1})\to0$ and $\rho^N_L(t,x^N_{k(N,x)-1})\to0$, which happens e.g. if $\inf_N|x^N_{k(N,x)+2}-x^N_{k(N,x)+1}|>0$ and $\inf_N|x^N_{k(N,x)}-x^N_{k(N,x)-1}|>0$. In particular, $\lim_N x^N_{k(N,x)+1}$ and $\lim_N x^N_{k(N,x)-1}$ are at the boundary of $\{\rho_L=0\}$ whenever $\hat{\Phi}_{v,L}\neq \hat c_L$.

Thus, by dominated convergence, one has that 
\begin{equation}
		\int_0^T \int_{\T_L}\rho^N_L(t,x)\frac{N}{c_L} \phi^\lin_v(\rho^N_L(t,x)) \partial_x \varphi(t,x) \dx \dt\to\quad\underset{N\to+\infty}{\longrightarrow}\quad - \int_0^T \int_{\T_L} \Phi_{v,L}(t,x)\partial_{xx} \varphi(t,x) \dx \dt,
\end{equation}
where $\Phi_{v,L}(t,x)=\phi_v(\rho_L(t,x))$ on a.e. $x\in\{\rho_L>0\}$.

\end{proof}	  

We now proceed to the proof of Theorem~\ref{thm:convinl}.

\begin{proof}[Proof of Theorem~\ref{thm:convinl}:] Let $ \rho_0>0$ a bounded $L^1$ density with unit mass and let 
	\begin{equation*}
		\hat{(\rho_0)}_L(x)=\rho_0(x)\chi_{[-L/2,L/2)}(x).
	\end{equation*} 
	Denote now by ${(\rho_0)}_L$ the $L$-periodic extension of $\hat{(\rho_0)}_L$ or its corresponding function on $\T_L$. Denoting by $\rho^N_L$ the deterministic particle approximations defined starting from ${(\rho_0)}_L$, we proved in the previous theorem that they converge, up to subsequences, to a bounded $L^1$ solution $\rho_L$ of the following PDE in weak form
	\begin{align}\label{eq:varphieq}
		\int_0^T \int_{\T_L}&  \rho_L(t,x) \partial_t \varphi(t,x) - \rho_{L}K' \ast \rho_L (t,x)v(\rho_L(t,x)) \partial_x \varphi(t,x) +\Phi_{v,L}(t,x)\partial_{xx}\varphi(t,x)\dx\dt=0,
	\end{align} 
	where $\varphi\in C^\infty_c((0,T)\times\T_L)$ and  $\Phi_{v,L}(t,x)=\phi_{v}(\rho_L(t,x))$ for a.e. $x\in\{\rho_L(t)>0\}$. 
	
	Let now $\psi\in C^\infty_c((0,T)\times\R)$. In particular, there exists $\bar L$ such that 
	\[
	\mathrm{spt}\psi\subset\!\subset [0,T]\times (-\bar L/2,\bar L/2).
	\] 
	For $L\geq\bar L$, denote by ${\psi}_L$ the $L$-periodic extension of $\psi$. Then, by~\eqref{eq:varphieq} one has that
	\begin{align}\label{eq:psieq}
		\int_0^T \int_{\T_L} & \rho_L(t,x) \partial_t \psi_L(t,x) - \rho_{L}K' \ast \rho_L (t,x)v(\rho_L(t,x)) \partial_x \psi_L(t,x) +\Phi_{v,L}(t,x)\partial_{xx}\psi_L(t,x)\dx \dt=0.
	\end{align} 
	Denoting by ${\rho}_L$ the $L$-periodic function on $\R$ corresponding to $\rho_L$ on $\T_L$ and setting  $\hat\rho_L={\rho}_L\chi_{[-L/2,L/2)}$ and $\hat\psi={\psi}\chi_{[-L/2,L/2)}$,~\eqref{eq:psieq} rewrites as  
	\begin{align*}
		\int_0^T \int_{\R} & \hat\rho_L(t,x) \partial_t \hat\psi(t,x) - \hat\rho_{L}(t,x)K' \ast \rho_L(t,x) v(\hat{\rho}_L(t,x)) \partial_x \hat\psi(t,x) +\Phi_{v,L}(t,x)\partial_{xx}\hat\psi(t,x)\dx \dt=0.	\end{align*}
	
	By Theorem~\ref{thm:compactness2} we know that, up to subsequences, $\hat\rho_L$ converge in $L^1([0,T]\times\R)$ to a bounded $L^1$ density $\rho$ as $L\to\infty$. In particular, as $L\to\infty$,
	
	\begin{equation*}
		\int_0^T \int_{\R}  \bigl(\hat\rho_L(t,x)-\rho(t,x)\bigr) \partial_t \hat\psi(t,x)+\bigl(\Phi_{v,L}(t,x)-\Phi_v(t,x)\bigr)\partial_{xx}\hat\psi(t,x)\dx \dt\quad\longrightarrow\quad0,
	\end{equation*}
where $\Phi_v(t,x)=\phi_v(\rho(t,x))$ for a.e. $x\in\{\rho(t)>0\}$.
%Moreover, since $\phi_v(\rho)\in BV_{\loc}(\R)$, $\phi_v(\rho_L)\to\phi_v(\rho)$ in $L^1$ and $\tilde{\Phi}_{v,L}$ concentrates on points $x$ s.t. either  $\lim_Nx^N_{k_N(x)+1}\in\partial\{\rho_L=0\}$ or $\lim_Nx^N_{k_N(x)-1}\in\partial\{\rho_L=0\}$, one has that $\exists \,\tilde{\Phi}_v$ s.t.  
%\begin{equation}
%	\int_0^T \int_{\T_L}\bigl[\tilde {\Phi}_{v,L}(t,x)-\tilde {\Phi}_{v}(t,x)\bigr]\partial_x\psi(t,x)\dx \dt\to0.
%\end{equation}
	
	Finally,
	\begin{equation*}
		\int_0^T \int_{\R} \bigl(\hat\rho_{L}(t,x)K' \ast\hat \rho_L (t,x)v(\hat{\rho}_L(t,x)) -\rho(t,x)K'\ast\rho(t,x)v({\rho}_L(t,x))\bigr)\partial_x \hat\psi(t,x)\dx\dt\quad\longrightarrow\quad0.
	\end{equation*}
	Thus we are left to show that
	\begin{equation*}
		\int_0^T \int_{\R} \hat\rho_{L}(t,x)K' \ast (\rho_L-\hat{\rho}_L) (t,x)v(\hat{\rho}_L(t,x))\partial_x \hat\psi(t,x)\dx\dt\quad\longrightarrow\quad0.
	\end{equation*}
	Using the fact that $\|\rho_L\|_{\infty}\leq\bar C$ and  that $K'\in L^1(\R)\cap L^\infty(\R)$ (see~\eqref{eq:k4}), one has that
	\begin{align*}
		\int_0^T \int_{\R} &|\hat\rho_{L}(t,x)|\int_{\R}|K'(x-y)| |\rho_L(t,y)-\hat{\rho}_L(t,y)| \dy|\partial_x \hat\psi(t,x)|\dx\dt\leq\notag\\
		&= \int_0^T \int_{\R} |\hat\rho_{L}(t,x)|\int_{\R\setminus[-L/2,L/2]}|K'(x-y)| |\rho_L(t,y)|\dy|\partial_x \hat\psi(t,x)|\dx\dt\notag\\
		&=J^L_1+J^L_2,
	\end{align*}
	with
	\begin{align*}
		J^L_1&=\int_0^T \int_{[-L/2,L/2]} |\hat\rho_{L}(t,x)|\int_{\R\setminus[-3L/4,3L/4]}|K'(x-y)| |\rho_L(t,y)|\dy|\partial_x \hat\psi(t,x)|\dx\dt,\notag\\
		J^L_2&=\int_0^T \int_{[-L/2,L/2]} |\hat\rho_{L}(t,x)|\int_{[-3L/4,-L/2]\cup[L/2,3L/4]}|K'(x-y)| |\rho_L(t,y)|\dy|\partial_x \hat\psi(t,x)|\dx\dt.
	\end{align*}
	On one hand,
	\begin{equation*}
		J^L_1\lesssim \|\hat\psi\|_{C^1}\int_{\{|z|>L/4\}}|K'(z)|\dz\quad\longrightarrow\quad0.
	\end{equation*}

	On the other hand,
	\begin{equation*}
		J^L_2\lesssim \|\hat\psi\|_{C^1}\|K'\|_{L^\infty}\int_{\{L/4<|y|<L/2\}}\hat \rho_L(y)\dy\quad\longrightarrow\quad0,
	\end{equation*}
	where we used the tightness of the measures $\hat{\rho}_L$ proved in Remark~\ref{rem:lambda}.
	
	Thus the proof of Theorem~\ref{thm:convinl} is concluded.
	
\end{proof}

\end{document}